\documentclass[12pt,reqno]{amsart}
\usepackage{amsmath}
\usepackage{amssymb}
\usepackage{amstext}
\usepackage{mathrsfs}
\usepackage{a4wide}
\usepackage{graphicx}
\usepackage{bbm}

\allowdisplaybreaks \numberwithin{equation}{section}
\usepackage{color}
\usepackage{cases}

\numberwithin{equation}{section}

\newtheorem{theorem}{Theorem}[section]

\newtheorem{lemma}[theorem]{Lemma}

\theoremstyle{definition}

\theoremstyle{remark}
\newtheorem{remark}[theorem]{Remark}

\newcommand{\Om}{\Omega}

\begin{document}

\title
[co-rotating and travelling annular patches]{Existence of co-rotating and travelling vortex patches with doubly connected components for active scalar equations}

\author{Daomin Cao, Guolin Qin,  Weicheng Zhan, Changjun Zou}

\address{Institute of Applied Mathematics, Chinese Academy of Sciences, Beijing 100190, and University of Chinese Academy of Sciences, Beijing 100049,  P.R. China}
\email{dmcao@amt.ac.cn}
\address{Institute of Applied Mathematics, Chinese Academy of Sciences, Beijing 100190, and University of Chinese Academy of Sciences, Beijing 100049,  P.R. China}
\email{qinguolin18@mails.ucas.edu.cn}
\address{Institute of Applied Mathematics, Chinese Academy of Sciences, Beijing 100190, and University of Chinese Academy of Sciences, Beijing 100049,  P.R. China}
\email{zhanweicheng16@mails.ucas.ac.cn}

\address{Institute of Applied Mathematics, Chinese Academy of Sciences, Beijing 100190, and University of Chinese Academy of Sciences, Beijing 100049,  P.R. China}
\email{zouchangjun17@mails.ucas.ac.cn}

%\thanks{This work is partially supported by ARC}

\begin{abstract}
	By applying implicit function theorem on contour dynamics, we prove the existence of co-rotating and travelling patch solutions for both Euler and the generalized surface quasi-geostrophic equation. The solutions obtained  constitute a desingularization of points vortices when the size of patch support vanishes. In particular, solutions constructed in this paper consist of doubly connected components, which is essentially different from all known results.
\end{abstract}

\maketitle

\section{Introduction}
We consider the following active scalar equation
\begin{align}\label{1-1}
	\begin{cases}
		\partial_t\theta+\mathbf{u}\cdot \nabla \theta =0&\text{in}\ \mathbb{R}^2\times (0,T)\\
		\ \mathbf{u}=\nabla^\perp(-\Delta)^{-1+\frac{\alpha}{2}}\theta     &\text{in}\ \mathbb{R}^2\times (0,T)
		\\
		\theta\big|_{t=0}=\theta_0 &\text{in}\ \mathbb{R}^2,
	\end{cases}
\end{align}
where $ 0\le\alpha<2$, $\theta(\boldsymbol x,t):\mathbb{R}^2\times (0,T)\to \mathbb{R}$ is the active scalar being transported by the velocity field $\mathbf{u}(\boldsymbol x,t):\mathbb{R}^2\times (0,T)\to \mathbb{R}^2$ generated by $\theta$, and $(x_1,x_2)^\perp=(x_2,-x_1)$. The operator $(-\Delta)^{-1+\frac{\alpha}{2}}$ is given by the expression
\begin{equation*}
	(-\Delta)^{-1+\frac{\alpha}{2}}\theta(\boldsymbol  x)=\int_{\mathbb{R}^2}K_\alpha(\boldsymbol x- \boldsymbol y)\theta(\boldsymbol y)d\boldsymbol y,
\end{equation*}
where $K_\alpha$ is the fundamental solution of $(-\Delta)^{-1+\frac{\alpha}{2}}$ in $\mathbb{R}^2$ given by
\begin{equation*}
K_\alpha(\boldsymbol  x )=\left\{
	\begin{array}{lll}
		\frac{1}{2\pi}\ln \frac{1}{|\boldsymbol x|}, \ \ \ \ \ \ \ \ \ \ \ \ \ \ \ \ \ \ \ \ \ & \text{if} \ \ \alpha=0;\\
		\frac{C_\alpha}{2\pi}\frac{1}{|\boldsymbol x |^\alpha}, \ \ \ C_\alpha=\frac{\Gamma(\alpha/2)}{2^{1-\alpha}\Gamma(\frac{2-\alpha}{2})}, & \text{if} \ \ \ 0<\alpha<2.
	\end{array}
	\right.
\end{equation*}
with $\Gamma$ the Euler gamma function.

When $\alpha=0$, \eqref{1-1} corresponds to the vorticity formulation of 2D incompressible Euler equation. In the case $\alpha=1$, \eqref{1-1} becomes the surface quasi-geostrophic (SQG) equation, which models geophysical flows and is applied to study the atmosphere circulation and ocean dynamics \cite{Hel, Juc, Lap}. The SQG equation also shares mathematical analogy with the three-dimensional incompressible Euler equation \cite{Con}. As a generalization of the Euler equation and the SQG equation, the gSQG model \eqref{1-1} with $0<\alpha<2$ was proposed by C\'{o}rdoba et al. in \cite{Cor}.

Mathematical study of \eqref{1-1} has attracted more and more attention in the past years. For the case $\alpha=0$, global well-posedness of solutions has already been solved in Sobolev spaces \cite{MB}. While Yudovich \cite{Yud} proved the existence and uniqueness of solutions to \eqref{1-1} with the initial data in $L^1\cap L^\infty$ in 1963. It is natural to ask whether it is possible to extend theories of the Euler equation to the general case $0<\alpha<2$ and a lot of effort has been made in recent years. Constantin et al. \cite{Con} established local well-posedness of the gSQG equation for classical solutions with sufficiently regular initial data. Local existence of solutions was also studied in different function spaces, see \cite{Chae0,Li0,Wu1,Wu2}.  Uniqueness and non-uniqueness of weak solutions under certain assumptions was recently shown in \cite{Cor1} and \cite{Buc} respectively. However, the problem whether the gSQG equation has global solutions or not remains unsolved. In \cite{Kis}, Kiselev and Nazarov constructed solutions of the gSQG equations with arbitrary Sobolev growth. Resnick \cite{Res} proved global existence of weak solutions to the SQG equations with any initial data in $L^2$, which was improved by Marchand \cite{Mar} to any initial data belonging to $L^p$ with $p>4/3$. To give concrete examples for global existence, various kind of global solutions to \eqref{1-1} were also constructed.

Due to the structure of the nonlinear term, all radially symmetric functions are stationary solutions to Euler and the gSQG equation. The first explicit non-trivial rotating solution of $2$-fold symmetry for $\alpha=0$ is constructed by Kirchhoff \cite{Kir}, which is an elliptic patch of semi-axes $a$ and $b$ subjected to a perpetual rotation with uniform angular velocity $ab/(a+b)^2$. Deem and Zabusky \cite{Deem} then carried out a series of numerical simulations, which provided evidence for existence of various $N$-fold symmetric solutions with $N\ge2$. In 1980s, Burbea \cite{Burb} outlined a new way of constructing $N$-fold symmetric patch solutions by bifurcation theory. Following Burbea's framework, Hassainia and Hmidi \cite{Has} used the contour dynamics equation to give a rigorous proof of existence for simply-connected $N$-fold symmetric V-states with $\alpha\in[0,1)$, which are bifurcated from the unit disk. Their construction relies on the structure of linearization of contour dynamics equation, and the uniform angular velocity is selected appropriately such that the transversality assumption of Crandall-Rabinowitz's theorem is satisfied. By changing the function spaces from H\"older space to $H^k$ space on torus, Castro et al. \cite{Cas1} extended this  result for the remaining open cases $\alpha\in[1,2)$ and showed the $C^\infty$ regularity of patch boundaries for $0<\alpha<2$.  In \cite{Cora,de1,de3,de2}, the existence of doubly connected V-states are established by studying the bifurcation from annulli with special angular velocity. However, there are two coupled nonlinear equations in this situation, and the estimate for the spectral of linearized operator is much more complicated. In \cite{Cas4,Hmi2}, $2$-fold symmetric V-states bifurcated from Kirchhoff elliptic vortices were considered for $\alpha=0$. Recently, \cite{ Cas3, Cas2,Car} successfully applied the same method to construct travelling and rotating smooth solutions. Using a different method which is based on variational principle, Gravejat and Smets \cite{ Gra} and  Godard-Cadillac \cite{Go0} constructed smooth traveling solutions for the SQG and gSQG equations respectively.

It is well-known that the Euler equation and the gSQG equation exhibit point vortex solutions \cite{Ros}. Of special interest are highly concentrated solutions which desingularize the point vortices. There are abundant literatures on the desingularized N-vortices problem for $\alpha=0$ and we refer to \cite{Bu0, Cao1, Cao4, CWZ, Dav, SV,T1, T2} and references therein. As for $0<\alpha<2$, the results are relatively less. The first result was obtained by Hmidi and Mateu \cite{HM}, where they showed the existence of concentrating travelling and co-rotating vortex pairs for $0\leq \alpha<1$ using the contour dynamics equation and implicit function theorem. Recently, Ao et al. \cite{Ao} gave the construction of concentrating smooth solutions for all $\alpha\in(0,2)$ by the Lyapunov-Schmidt reduction method. Godard-Cadillac et al. \cite{Go} constructed concentrating co-rotating $m$-fold symmetric vortex patches for $\alpha\in(0,2)$ via vorticity method.

It is noteworthy that all known concentrating solutions have supports composed of components which are perturbations of small disks. The aim of our paper is to construct co-rotating and travelling patch solutions for $0\leq \alpha<2$, whose supports consist of doubly connected components perturbed from small annuli. Recall that a planar domain $D_0$ is called doubly connected if $D_0=D_2\setminus D_1$ for simply connected bounded domains $D_1$ and $D_2$ satisfying $\overline{D}_1\subset D_2$. Our solutions provide a brand new phenomenon for the behavior of quasi-geostrophic flow.

The first kind of global solutions we are concerned with in this paper are co-rotating solutions. Suppose that the initial data $\theta_{0,\varepsilon}$ is of $N$-fold symmetric patch type, that is
\begin{equation}\label{1-2}
	\theta_{0,\varepsilon}( x)=\frac{1}{\pi(1-b^2+\gamma b^2)\varepsilon^2}\sum\limits_{n=0}^{N-1}\omega_n^\varepsilon,
\end{equation}
where $\omega_0^\varepsilon=\boldsymbol\chi_{D_{out}^\varepsilon} +(\gamma-1)\boldsymbol\chi_{D_{in}^\varepsilon}$ for some $b\in(0,1)$, $\gamma\not= \frac{b^2-1}{b2}$ and $\omega_n^\varepsilon$ satisfy
\begin{equation*}
	\omega_n^\varepsilon(\boldsymbol x)=\omega_0^\varepsilon(Q_{\frac{2\pi n}{N}}\left(\boldsymbol x-d\boldsymbol e_1\right)+d\boldsymbol e_1),\quad \forall \,n=1,2,\ldots N-1
\end{equation*}with some $d>0$ fixed, $\boldsymbol e_1=(1,0)$ the unit vector in $x_1$ direction and  $Q_\beta$ to denote the counterclockwise rotation operator of angle $\beta$ with respect to the origin.
In addition, $D_{\text{out}}^\varepsilon$ and  $D_{\text{in}}^\varepsilon$ are simply connected domains satisfying $\overline{D}_{\text{in}}^\varepsilon\subset D_{\text{out}}^\varepsilon $ and are perturbations of small disks $\varepsilon B_1(0)$ and $\varepsilon B_b(0)$  respectively. Notice that if $\gamma=0$, then  $\omega_0^\varepsilon=\boldsymbol\chi_{D_{\text{out}}^\varepsilon\setminus D_{\text{in}}^\varepsilon}$ and $D_{\text{out}}^\varepsilon\setminus D_{\text{in}}^\varepsilon$ is a doubly connected domain. For $\gamma\ll 0$, an interesting phenomenon appears  where $\omega_1^\varepsilon<0$ in $D_{\text{out}}\setminus D_{\text{in}}$, and $\omega_1^\varepsilon>0$ in $D_{\text{in}}$. The case $\gamma=1$ corresponds to simply connected cases, which have been studied as mentioned above. We intend to prove the existence of a series of co-rotating $N$-fold symmetric solutions to \eqref{1-1} about $(d,0)$, which take the form
\begin{equation*}
	\theta_\varepsilon( \boldsymbol x,t)=\theta_{0,\varepsilon}\left(Q_{\Omega t}( \boldsymbol x-d\boldsymbol e_1)+d\boldsymbol e_1\right)
\end{equation*}
with $\Omega$ some fixed angular velocity.

As mentioned above, Hmidi and Mateu \cite{HM} obtained the existence of concentrating travelling and co-rotating vortex pairs with simply connected components for $0\leq \alpha<1$ using the contour dynamics equation and implicit function theorem. To sketch their proof, we denote $G(\varepsilon, \Om, f)=0$ as the contour dynamics equation that the patch boundary satisfies, where $f$ parameterizes the boundary. The first step is to extend $G(\varepsilon, \Om, f)$ such that $\varepsilon\leq 0$ is allowed. One can verify the case $\varepsilon=0$, $f=0$ corresponds to the point vortices and hence $G(0, \Om, 0)=0$ holds for some special $\Om=\Om_{sing}$. To apply the implicit function theorem, the Gautex derivative $\partial_f G(0, \Om, 0)=0$ should be an isomorphism, which only holds true for the space where the first Fourier coefficient vanishes. Consequently, the condition $\int G(\varepsilon, \Om, f) \sin(x)dx=0$ is needed, which can be achieved by choosing proper
$\Om(\varepsilon,f)=\Om_{sing}+O(\varepsilon)$. Then the existence of solutions follows from the implicit function theorem.

However, for the doubly connected components case, there are two coupled equations $G_1(\varepsilon, \Om, f_1,f_2)=0$ and $G_2(\varepsilon, \Om, f_1,f_2)=0$ to be solved. It seems impossible to choose suitable $\Om(\varepsilon)$ such that $\int G_1(\varepsilon, \Om, f_1,f_2)\sin(x)dx=0$ and $\int G_2(\varepsilon, \Om, f_1,f_2)\sin(x)dx=0$ hold simultaneously. Fortunately, we find that letting $$\int G_1(\varepsilon, \Om, f_1,f_2)\sin(x)dx =(1-\gamma)b^2 \int G_2(\varepsilon, \Om, f_1,f_2)\sin(x)dx$$ is enough for our purpose if we take appropriate function spaces. By choosing $\Om$ such that this assumption is satisfied, we successfully apply  implicit function theorem to prove the following  result of existence for the Euler equation.
\begin{theorem}\label{thm1}
	Suppose $\alpha=0$,  $N\ge 2$ and  $\gamma\in \mathbb{R}\setminus\{0\}$. Then there exist $\varepsilon_0>0$ and $b\in(0,1)$,  such that for any $\varepsilon\in (0,\varepsilon_0)$ , equation \eqref{1-1} has a global co-rotating solution $\theta_\varepsilon(\boldsymbol x-d\boldsymbol e_1,t)=\theta_{0,\varepsilon}(Q_{\Omega^0 t}(\boldsymbol x-d\boldsymbol e_1))$ centered at $(d,0)$, where $\theta_{0,\varepsilon}$ is of the form \eqref{1-2}, and $\Omega^0$ satisfies
	$$\Omega^0=\Omega_*^0+O(\varepsilon)$$
	with $\Omega_*^0$ given by
	\begin{equation}\label{1-4}
		\Omega_*^0:=\sum_{n=1}^{N-1} \frac{1-\cos(\frac{2\pi n}{N})}{2\pi \left((-1+\cos(\frac{2\pi n}{N}))^2 +\sin^2(\frac{2\pi n}{N})\right)d^{2}}.
	\end{equation}
    Moreover, it holds in the sense of measures
     \begin{equation*}
    	\theta_{0,\varepsilon}(x)\rightarrow \sum_{n=0}^{N-1}\delta_{Q_{\frac{2n\pi}{N}}(d,0)}\ \ \text{as}\ \ \varepsilon\to 0^+.
    \end{equation*}
\end{theorem}
\begin{remark}
	If we take $\mathcal{B}$ as the set of all $b$ such that Theorem \ref{thm1} holds true. Then $(0,1)\setminus \mathcal{B}$ composed of at most countable many points.  Moreover, for $\gamma\in(0,1)$ it holds $(0,\frac{\sqrt 2}{2})\subset \mathcal{B}$ and if $\gamma\geq1$ then $\mathcal{B}=(0,1)$. Theorem \ref{thm1} does not cover the case $\gamma=0$ due to the fail of the invertibility of derivative matrices, see Lemma \ref{lm3-5} for details. However, since the bifurcation method does not rely on this invertibility, in \cite{de2} the authors let $\gamma=0$ and constructed doubly connected vorticity to the Euler equation.
\end{remark}

We also construct co-rotating solutions of the form \eqref{1-2} for the gSQG equation. Since the linearized operator has better property than the Euler case ($\alpha=0$), we can take $\gamma=0$ here, which means that the solutions constructed have supports composed of doubly connected components.

\begin{theorem}\label{thm1'}
	Suppose $\alpha\in (0,2)$,  $N\ge 2$ and $\gamma\in[0,1)$. Then there exist $\varepsilon_0>0$ and $b_0\in(0,1)$  such that for any $\varepsilon\in (0,\varepsilon_0)$ and $b\in(0,b_0)$, equation \eqref{1-1} has a global co-rotating solution $\theta_\varepsilon(\boldsymbol x-d\boldsymbol e_1,t)=\theta_{0,\varepsilon}(Q_{\Omega^\alpha t}(\boldsymbol x-d\boldsymbol e_1))$ centered at $(d,0)$, where $\theta_{0,\varepsilon}$ is of the form \eqref{1-2}, and $\Omega^\alpha$ satisfies
	$$\Omega^\alpha=\Omega_*^\alpha+O(\varepsilon^{1+\alpha})$$
	with $\Omega_*^\alpha$ given by
	\begin{equation}\label{1-4'}
		\Omega_*^\alpha:=\sum_{n=1}^{N-1} \frac{\alpha C_\alpha(1-\cos(\frac{2\pi n}{N}))}{2\pi\left((-1+\cos(\frac{2\pi n}{N}))^2 +\sin^2(\frac{2\pi n}{N})\right)^{1+\frac{\alpha}{2}}d^{2+\alpha}},\quad 0<\alpha<2.
	\end{equation}
	Moreover, it holds in the sense of measures
	\begin{equation*}
		\theta_{0,\varepsilon}(x)\rightarrow \sum_{n=0}^{N-1}\delta_{Q_{\frac{2n\pi}{N}}(d,0)}\ \ \text{as}\ \ \varepsilon\to 0^+.
	\end{equation*}
\end{theorem}
Our next result concerns the travelling patch pairs. In this situation, the initial data $\theta_{0,\varepsilon}$ is given by
\begin{equation}\label{1-5}
	\theta_{0,\varepsilon}(\boldsymbol x)=\frac{1}{\pi(1-b^2+\gamma b^2)\varepsilon^2}(\omega_0^\varepsilon-\omega_1^\varepsilon).
\end{equation}
where $\omega_0^\varepsilon=\boldsymbol\chi_{D_{\text{out}}^\varepsilon} +(\gamma-1)\boldsymbol\chi_{D_{\text{in}}^\varepsilon}$ for some $\gamma\in [0,1]$, $b\in(0,1)$ and $\omega_1^\varepsilon$ satisfies
\begin{equation*}
	\omega_1^\varepsilon(x_1, x_2)=\omega_0^\varepsilon(-x_1+2d, x_2)
\end{equation*}with some $d>0$ fixed, and $\boldsymbol e_1$ the unit vector in $x_1$ direction.
Furthermore, $D_{\text{out}}^\varepsilon$ and  $D_{\text{in}}^\varepsilon$ are simply connected domains satisfying $\overline{D}_{\text{in}}^\varepsilon\subset D_{\text{out}}^\varepsilon $ and are perturbations of the small disks $\varepsilon B_1(0)$ and $\varepsilon B_b(0)$  respectively.
  A travelling patch pair centered at $(d,0)$ takes the form
\begin{equation*}
	\theta_\varepsilon(x,t)=\theta_{0,\varepsilon}( x-tW\boldsymbol e_2)
\end{equation*}
with $W$ some fixed speed, and $\boldsymbol e_2$ the unit vector in $x_2$ direction. Using the same method as the proof of Theorems \ref{thm1} and \ref{thm1'}, we obtain the following result.
\begin{theorem}\label{thm2}
	Suppose $\alpha\in [0,2)$, $\gamma\in \mathbb{R}\setminus\{0\}$ if $\alpha=0$ and $\gamma\in [0,1)$ if $\alpha\in(0,2)$. Then there exist $\varepsilon_0>0$ and $b\in(0,1)$  such that for any $\varepsilon\in (0,\varepsilon_0)$, equation \eqref{1-1} has a global travelling solution $\theta_\varepsilon(\boldsymbol x,t)=\theta_{0,\varepsilon}(\boldsymbol x-tW^\alpha\boldsymbol e_2)$ , where $\theta_{0,\varepsilon}$ is of the form \eqref{1-5}, and $W^\alpha$ satisfies
	$$W^\alpha=W^\alpha_*+O(|\varepsilon|^{1+\alpha})$$
	with $W^\alpha_*$ given by
	\begin{equation}\label{1-6}
		W^\alpha_*:=\begin{cases}\frac{1}{4\pi d},&\quad \alpha=0,\\ \frac{\alpha C_\alpha}{2\pi (2d)^{1+\alpha}},&\quad 0<\alpha<2.\end{cases}
	\end{equation}
	Moreover, it holds in the sense of measures
	\begin{equation*}
		\theta_{0,\varepsilon}(x)\rightarrow \delta_{(0,0)}-\delta_{(2d,0)}\ \ \text{as}\ \ \varepsilon\to 0^+.
	\end{equation*}
\end{theorem}

\begin{remark}\label{rem2}
	If we take $\mathcal{B}$ as the set of all $b$ such that Theorem \ref{thm2} holds true. Then for $\alpha=0$, $(0,1)\setminus \mathcal{B}$ composed of at most countable many points.  Moreover, for $\gamma\in(0,1)$ it holds $(0,\frac{\sqrt 2}{2})\subset \mathcal{B}$ and if $\gamma\geq1$ then $\mathcal{B}=(0,1)$. For $\alpha\in(0,2)$, there exists $b_0\in(0, 1)$ such that $(0,b_0)\subset \mathcal{B}$.
\end{remark}

\begin{remark}\label{rem3}
	Calculating the signed curvature as the Corollary at the end of \cite{Hmi} and  Proposition 4.12 in \cite{Cas1}, one can prove that the doubly connected domain $D_0^\varepsilon$ in Theorem \ref{thm1}, \ref{thm1'} or \ref{thm2} is the difference of two convex sets. We also expect that the regularity of the boundaries can be improved to be  analytic using related spaces defined in \cite{Cas4}. Some other interesting prosperities for doubly connected vortices with single component have been investigated in \cite{Hmi1}.
\end{remark}

This paper is organized as follows. In section 2, we derive the contour dynamics equation and define the function spaces used later. In Section 3, we show the existence of $N$-fold symmetric co-rotating patch solutions for $\alpha=0$.  Section 4 is devoted to the existence of $N$-fold symmetric co-rotating patch solutions with  doubly connected components for $0<\alpha<2$. In Section 5, we sketch the proof of existence of travelling patch pair solutions with  doubly connected components.

\section{The contour dynamics equation and function spaces}

The main tool used in this paper is the contour dynamics equation for $\alpha$-patch. An $\alpha$-patch is a solution of \eqref{1-1} whose initial data is given by $\theta_0=\kappa \boldsymbol \chi_{D}$, with $\kappa\in \mathbb{R}$, $D\subset \mathbb{R}^2$ a bounded domain and $\boldsymbol \chi_{D}$ its characteristic function. Due to the transport formula $\partial_t\theta+\mathbf{u}\cdot \nabla \theta =0$, the solution will preserve its patch structure and can be written as $\theta_t=\kappa\boldsymbol \chi_{D_t}$. When $0\leq\alpha<1$, by using Biot-Savart law and Green-Stokes formula, the velocity can be recovered by
\begin{equation*}
	\boldsymbol u(\boldsymbol x,t)=\begin{cases}\frac{\kappa}{2\pi}\int_{\partial D_t}\log\left(\frac{1}{|\boldsymbol x-\boldsymbol y|}\right)d\boldsymbol y,&\quad \alpha=0,\\
\vspace{0.1cm}\\
\frac{\kappa C_\alpha}{2\pi}\int_{\partial D_t}\frac{1}{|\boldsymbol x-\boldsymbol y|^\alpha}d\boldsymbol y, &\quad 0<\alpha<1.\end{cases}
\end{equation*}
Thus if the patch boundary $\partial D_t$ is parameterized as $\boldsymbol z(t,\sigma)$ with $\sigma\in [0,2\pi)$, then $\boldsymbol z(t,\sigma)$ satisfies
\begin{equation*}
	\partial_t \boldsymbol z(t,\sigma)=\begin{cases} \frac{\kappa}{2\pi}\int_0^{2\pi}\log\left(\frac{1}{|\boldsymbol z(t,\sigma)-\boldsymbol z(t,\tau)|}\right)\partial_\tau\boldsymbol z(t,\tau)d\tau,&\quad \alpha=0,\\
\vspace{0.1cm}\\
\frac{\kappa C_\alpha}{2\pi}\int_0^{2\pi}\frac{\partial_\tau\boldsymbol z(t,\tau)}{|\boldsymbol z(t,\sigma)-\boldsymbol z(t,\tau)|^\alpha}d\tau&\quad 0<\alpha<1, \end{cases}
\end{equation*}
which is known as the contour dynamics equation. This equation is locally well-posed if the boundary of the initial $\alpha$-patch is composed with finite number of disjoint smooth Jordan curves.
For $1\leq \alpha<2$ the above integral is divergent.  To eliminate the singularity, one can change the velocity at the boundary by subtracting a tangential vector and define
\begin{equation*}
	\partial_t \boldsymbol z(t,\sigma)= \frac{\kappa C_\alpha}{2\pi}\int_0^{2\pi}\frac{\partial_\tau\boldsymbol z(t,\tau)-\partial_\sigma\boldsymbol z(t,\sigma)}{|\boldsymbol z(t,\sigma)-\boldsymbol z(t,\tau)|^\alpha}d\tau.
\end{equation*}

The interfaces of the  $\alpha$-patch under our consideration can be parameterized as two $2\pi$ periodic curves $\boldsymbol z_i(x)$ with $i=1,2$ and $x\in [0, 2\pi)$. For $0<b<1$, we assume that $\boldsymbol z_i$ are of the following form:
$$\boldsymbol z_i(x,t)=\left(\varepsilon R_i(x,t)\cos(x), \varepsilon R_i(x,t)\sin(x)\right),$$
with
$$R_1(x,t)=1+\varepsilon^{1+\alpha}f_1(x,t),$$
$$R_2(x,t)=b+\varepsilon^{1+\alpha} f_2(x,t).$$
For simplicity, let
$$\int\!\!\!\!\!\!\!\!\!\; {}-{} g(\tau)d\tau:=\frac{1}{2\pi}\int_0^{2\pi}g(\tau)d\tau$$
be the mean value of integral on the unit circle and denote $\boldsymbol{f}=(f_1,f_2)$.

For co-rotating patches
\begin{equation*}
	\theta_\varepsilon(\boldsymbol x,t)=\theta_{0,\varepsilon}\left(Q_{\Omega t}(\boldsymbol x-d\boldsymbol e_1)+d\boldsymbol e_1\right), \quad \boldsymbol x\in \mathbb{R}^2
\end{equation*}
with $\Omega$ some uniform angular velocity. Inserting this equality into \eqref{1-1}, we derive
\begin{equation*}
	\left(\mathbf{u}_0(\boldsymbol x)+\Omega(\boldsymbol x-d\boldsymbol e_1)^\perp\right)\cdot \nabla \theta_{0,\varepsilon}(\boldsymbol x)=0.
\end{equation*}
Then we  use the patch structure to obtain
\begin{equation*}
	\left(\mathbf{u}_0(\boldsymbol x)+\Omega(\boldsymbol x-d\boldsymbol e_1)^\perp\right)\cdot \mathbf n(\boldsymbol x)=0, \ \ \ \forall \, \boldsymbol x \in \cup_{n=0}^{N-1}\partial D^\varepsilon_n,
\end{equation*}
where $D_n^\varepsilon$ is the rotation of annuler $D_{out}^\varepsilon \setminus \overline{D}_{in}^\varepsilon$ and $\mathbf n(\boldsymbol x)$ is the normal vector to the boundary. According to Biot-Savart law and Green-Stokes formula, the interfaces satisfies the following equation for $\alpha=0$
\begin{align}\label{2-1}
	0&= G_i^0(\varepsilon, \Omega, \boldsymbol{f})\\
	&=\pi(1-b^2+\gamma b^2)\Omega\left(\varepsilon R_i'(x)-\frac{dR_i'(x)\cos(x)}{R_i(x)}+d\sin(x)\right)\nonumber\\
	&+\frac{1}{2 R_i(x)\varepsilon}\int\!\!\!\!\!\!\!\!\!\; {}-{} \log\left(\frac{1}{ \left(R_i(x)-R_1(y)\right)^2+4R_i(x)R_1(y)\sin^2\left(\frac{x-y}{2}\right)}\right)\nonumber\\
	&\quad\times \left[(R_i(x)R_1(y)+R'_i(x)R'_1(y))\sin(x-y)+(R_i(x)R'_1(y)-R'_i(x)R_1(y))\cos(x-y)\right] dy\nonumber\\
	&+ \frac{\gamma-1}{ 2 R_i(x)\varepsilon}\int\!\!\!\!\!\!\!\!\!\; {}-{}\log\left(\frac{1}{ \left(R_i(x)-R_2(y)\right)^2+4R_i(x)R_2(y)\sin^2\left(\frac{x-y}{2}\right)}\right)\nonumber\\
	&\quad\times \left[(R_i(x)R_2(y)+R'_i(x)R'_2(y))\sin(x-y)+(R_i(x)R'_2(y)-R'_i(x)R_2(y))\cos(x-y)\right]dy\nonumber\\
	&+\sum_{n=1}^{N-1} \frac{1}{ R_i(x) \varepsilon} \int\!\!\!\!\!\!\!\!\!\; {}-{}\log\left( \frac{1}{\left| \boldsymbol z_i(x)-d\boldsymbol e_1- Q_{\frac{2\pi n}{N}} \left(\boldsymbol z_1(y)-d\boldsymbol e_1\right)\right|}\right)\nonumber\\
	&\quad \times\left[(R_i(x)R_1(y)+R_i'(x)R_1'(y))\sin(x-y-\frac{2\pi n}{N})\right.\nonumber\\
	&\qquad  \left.+(R_i(x)R_1'(y)-R_i'(x)R_1(y))\cos(x-y-\frac{2\pi n}{N})\right] dy\nonumber\\
	&+\sum_{n=1}^{N-1} \frac{\gamma-1}{ R_i(x) \varepsilon} \int\!\!\!\!\!\!\!\!\!\; {}-{}\log\left( \frac{1}{\left| \boldsymbol z_i(x)-d\boldsymbol e_1- Q_{\frac{2\pi n}{N}} \left(\boldsymbol z_2(y)-d\boldsymbol e_1\right)\right|}\right)\nonumber\\
	&\quad \times\left[(R_i(x)R_2(y)+R_i'(x)R_2'(y))\sin(x-y-\frac{2\pi n}{N})\right.\nonumber\\
	&\qquad\,\left.+(R_i(x)R_2'(y)-R_i'(x)R_2(y))\cos(x-y-\frac{2\pi n}{N})\right] dy\nonumber\\
	&=G_{i1}+G_{i2}+G_{i3}+G_{i4}+G_{i5}.\nonumber
\end{align}

When $0< \alpha<2$, the interfaces obey the following equation
\begin{align}\label{2-2}
	0&= G_i^\alpha(\varepsilon, \Omega, \boldsymbol{f})\\
	&=\pi(1-b^2+\gamma b^2)\Omega\left(\varepsilon R_i'(x)-\frac{dR_i'(x)\cos(x)}{R_i(x)}+d\sin(x)\right)\nonumber\\
	&+\frac{C_\alpha}{R_i(x)\varepsilon|\varepsilon|^\alpha}\int\!\!\!\!\!\!\!\!\!\; {}-{} \frac{1}{\left| \left(R_i(x)-R_1(y)\right)^2+4R_i(x)R_1(y)\sin^2\left(\frac{x-y}{2}\right)\right|^{\frac{\alpha}{2}}}\nonumber\\
	&\quad \times \left[(R_i(x)R_1(y)+R'_i(x)R'_1(y))\sin(x-y)+(R_i(x)R'_1(y)-R'_i(x)R_1(y)\cos(x-y))\right]dy\nonumber\\
	& +\frac{(\gamma-1)C_\alpha}{ R_i(x)\varepsilon|\varepsilon|^\alpha}\int\!\!\!\!\!\!\!\!\!\; {}-{}\frac{1}{\left| \left(R_i(x)-R_2(y)\right)^2+4R_i(x)R_2(y)\sin^2\left(\frac{x-y}{2}\right)\right|^{\frac{\alpha}{2}}}\nonumber\\
	&\quad\times \left[(R_i(x)R_2(y)+R'_i(x)R'_2(y))\sin(x-y)+(R_i(x)R'_2(y)-R'_i(x)R_2(y))\cos(x-y)\right]dy\nonumber\\
	&+\sum_{n=1}^{N-1} \frac{C_\alpha}{ R_i(x) \varepsilon} \int\!\!\!\!\!\!\!\!\!\; {}-{} \frac{1}{\left| \boldsymbol z_i(x)-d\boldsymbol e_1- Q_{\frac{2\pi n}{N}} \left(\boldsymbol z_1(y)-d\boldsymbol e_1\right)\right|^{\alpha}}\nonumber\\
	&\quad\times\left[(R_i(x)R_1(y)+R_i'(x)R_1'(y))\sin(x-y-\frac{2\pi n}{N})\right.\nonumber\\
	&\qquad \left. +(R_i(x)R_1'(y)-R_i'(x)R_1(y))\cos(x-y-\frac{2\pi n}{N})\right] dy\nonumber\\
	&+\sum_{n=1}^{N-1} \frac{(\gamma-1)C_\alpha}{ R_i(x) \varepsilon} {\int\!\!\!\!\!\!\!\!\!\; {}-{}} \frac{1}{\left| \boldsymbol z_i(x)-d\boldsymbol e_1- Q_{\frac{2\pi n}{N}} \left(\boldsymbol z_2(y)-d\boldsymbol e_1\right)\right|^{\alpha}}\nonumber\\
	&\quad\times\left[(R_i(x)R_2(y)+R_i'(x)R_2'(y))\sin(x-y-\frac{2\pi n}{N})\right.\nonumber\\
	&\qquad\left. +(R_i(x)R_2'(y)-R_i'(x)R_2(y))\cos(x-y-\frac{2\pi n}{N})\right] dy\nonumber
\end{align}

For $0\leq\alpha<1$, the function spaces used in the current paper are
\begin{equation*}
	X^k=\left\{ g\in H^k, \ g(x)= \sum\limits_{j=1}^{\infty}a_j\cos(jx)\right\},
\end{equation*}

$$X_b^{k}:=\Bigg\{(f_1, f_2)\in X^k\times X^k \Big| \int\!\!\!\!\!\!\!\!\!\; {}-{} f_1(x)\cos(x)dx=(1-\gamma)b^2 \int\!\!\!\!\!\!\!\!\!\; {}-{} f_2(x)\cos(x)dx\Bigg\},$$

\begin{equation*}
	Y^{k}=\left\{ g\in H^{k}, \ g(x)= \sum\limits_{j=1}^{\infty}a_j\sin(jx)\right\},
\end{equation*}
and
$$Y_b^{k}:=\Bigg\{(g_1, g_2)\in Y^{k}\times Y^{k}\Big| \int\!\!\!\!\!\!\!\!\!\; {}-{} g_1(x)\sin(x)dx=(1-\gamma)b^2 \int\!\!\!\!\!\!\!\!\!\; {}-{} g_2(x)\sin(x)dx\Bigg\}.$$
The following spaces are also in order for the case $\alpha=1$
\begin{equation*}
	\begin{split}
		X^{k+\log}=&\left\{ g\in H^k, \ g(x)= \sum\limits_{j=2}^{\infty}a_j\cos(jx), \ \left\|\int_0^{2\pi}\frac{\partial^kg(x-y)-\partial^kg(x)}{|\sin(\frac{y}{2})|}dy\right\|_{L^2}<\infty \right\},\\
	\end{split}
\end{equation*}
and
$$X_b^{k+\log}:=\Bigg\{(f_1, f_2)\in X^{k+\log}\times X^{k+\log} \Big| \int\!\!\!\!\!\!\!\!\!\; {}-{} f_1(x)\cos(x)dx=(1-\gamma)b^2 \int\!\!\!\!\!\!\!\!\!\; {}-{} f_2(x)\cos(x)dx\Bigg\}.$$

For $1<\alpha<2$, we will need the following fractional spaces
\begin{equation*}
	\begin{split}
		X^{k+\alpha-1}=&\left\{ g\in H^k, \ g(x)= \sum\limits_{j=2}^{\infty}a_j\cos(jx), \ \left\|\int_0^{2\pi}\frac{\partial^kg(x-y)-\partial^kg(x)}{|\sin(\frac{y}{2})|^\alpha}dy\right\|_{L^2}<\infty \right\},\\
	\end{split}
\end{equation*}

and
$$X_b^{k+\alpha-1}:=\Bigg\{(f_1, f_2)\in X^{k+\alpha-1}\times X^{k+\alpha-1} \Big| \int\!\!\!\!\!\!\!\!\!\; {}-{} f_1(x)\cos(x)dx=(1-\gamma)b^2 \int\!\!\!\!\!\!\!\!\!\; {}-{} f_2(x)\cos(x)dx\Bigg\}.$$
The norm for $X^k$ and $Y^{k}$ is the $H^k$-norm. For $X^{k+\log}$ and $X^{k+\alpha-1}$, the norms are given by the sum of the $H^k$-norm and the integral in their definitions respectively.
Note that for every $\beta>0$, the embedding $X^{k+\beta}\subset X^{k+\log}\subset X^{k}$ holds.

Throughout this paper, we always assume $k\geq 3$.

\section{Existence of co-rotating solutions: the case $\alpha=0$}

We will split up our proof into several lemmas, which correspond to checking the hypotheses of the implicit function theorem.
Set $V:=\{f\in X_b^k\mid ||g||_{H^k}<1\}$. In what follows, $\varepsilon_0$ denotes a small positive number determined in the process of proof.

We first prove that $G^0_i$ is continuous and can be extended to $\varepsilon\leq 0$.
\begin{lemma}\label{lm3-1}
	$G_i^0(\varepsilon, \Omega, \boldsymbol f): \left(-\varepsilon_0, \varepsilon_0\right)\times \mathbb{R} \times V \rightarrow Y^{k-1}$ is continuous for $i=1,2$.
\end{lemma}
\begin{proof}
	We only prove the continuity of $G_1^0$.
	
	For the first term $G_{11}$, it is easy to show that $G_{11}: \left(-\frac{1}{2}, \frac{1}{2}\right)\times \mathbb{R} \times V \rightarrow Y^{k-1}$ is continuous by defining $G_{11}(0, \Om, \boldsymbol{f})=\pi(1-b^2+\gamma b^2)\Om d\sin(x)$. And we can rewrite
	$G_{11}$ as
	\begin{equation}\label{3-1}
		G_{11}=\pi(1-b^2+\gamma b^2)\Omega \left(d\sin(x)+\varepsilon\mathcal{R}_{11}(\varepsilon, f_1)\right),
	\end{equation}
	where $\mathcal{R}_{11}(\varepsilon,\Om, f_1): \left(-\frac{1}{2}, \frac{1}{2}\right)\times \mathbb{R} \times X^k \rightarrow Y^{k-1}$ is continuous.
	
	Now, consider the term $G_{12}$. We will prove that $\partial^l G_{12}\in L^2$ for $l=0,1, \ldots, k-1$. Since $R_1(x)=1+\varepsilon f_1(x)$, the possible singularity caused by $\varepsilon=0$ may occur only when we take zeroth order derivative of $G_{12}$. Thus, we first show that $G_{12}\in L^2$.

	Firstly, we consider the terms in $G_{12}$,
	\begin{align*}
		&\quad\frac{1}{2 R_1(x)\varepsilon}\int\!\!\!\!\!\!\!\!\!\; {}-{} \log\left(\frac{1}{ \left(R_1(x)-R_1(y)\right)^2+4R_1(x)R_1(y)\sin^2\left(\frac{x-y}{2}\right)}\right)\nonumber\\
		&\qquad\times \left[R'_1(x)R'_1(y)\sin(x-y)+(R_1(x)R'_1(y)-R'_1(x)R_1(y))\cos(x-y)\right] dy\nonumber\\
		&=\frac{1}{2 R_1(x)}\int\!\!\!\!\!\!\!\!\!\; {}-{} \log\left(\frac{1}{ \left(R_1(x)-R_1(y)\right)^2+4R_1(x)R_1(y)\sin^2\left(\frac{x-y}{2}\right)}\right)\nonumber\\
		&\qquad\times \left[\varepsilon f'_1(x)f'_1(y)\sin(x-y)+(R_1(x)f'_1(y)-f'_1(x)R_1(y))\cos(x-y)\right] dy\nonumber\\
		&=\frac{1}{2}\int\!\!\!\!\!\!\!\!\!\; {}-{} \log\left(\frac{1}{ \left(R_1(x)-R_1(y)\right)^2+4R_1(x)R_1(y)\sin^2\left(\frac{x-y}{2}\right)}\right)(f_1'(y)-f_1'(x))\cos(x-y)dy\\
		&+\frac{\varepsilon}{2 R_1(x)}\int\!\!\!\!\!\!\!\!\!\; {}-{} \log\left(\frac{1}{ \left(R_1(x)-R_1(y)\right)^2+4R_1(x)R_1(y)\sin^2\left(\frac{x-y}{2}\right)}\right)\\
		&\quad \times\left[ f'_1(x)f'_1(y)\sin(x-y)+f'_1(x)(f_1(x)-f_1(y))\cos(x-y)\right] dy.
	\end{align*}
		Since $c|x-y|^2\leq \left|\left(R_1(x)-R_1(y)\right)^2+4R_1(x)R_1(y)\sin^2\left(\frac{x-y}{2}\right)\right|\leq C|x-y|^2$ for some positive constants $c$ and $C$, we deduce that the kernel
	\begin{equation}\label{3-2}
\log\left(\frac{1}{ \left(R_1(x)-R_1(y)\right)^2+4R_1(x)R_1(y)\sin^2\left(\frac{x-y}{2}\right)}\right)\in L^1.
	\end{equation}
Hence, these terms belong to $L^2$.
	
	 To prove that the remaining term in $G_{12}$ belongs to $L^2$, we will use the following Taylor's formula:
	\begin{equation}\label{3-3}
		\log\left(\frac{1}{A+B}\right)=\log\left(\frac{1}{A}\right)-\int_0^1\frac{B}{A+tB}dt.
	\end{equation}and	
    \begin{equation}\label{3-4}
	\frac{1}{(A+B)^\beta}=\frac{1}{A^\beta}-\beta\int_0^1\frac{B}{(A+tB)^{1+\beta}}dt.
    \end{equation}
	Denote
	$$A:=4\sin^2\left(\frac{x-y}{2}\right),$$
	and $$B:=\varepsilon(f_1(x)-f_1(y))^2+\sin^2\left(\frac{x-y}{2}\right)\left(4f_1(x)+4f_1(y)+\varepsilon f_1(x)f_1(y)\right).$$
	Then, there holds $$\left(R_1(x)-R_1(y)\right)^2+4R_1(x)R_1(y)\sin^2\left(\frac{x-y}{2}\right)=A+\varepsilon B.$$
	
	Using \eqref{3-3}, \eqref{3-4} and the fact that $\sin(\cdot)$ is an odd function, we have
	\begin{equation*}
		\begin{split}
			&\quad\frac{1}{\varepsilon}\int\!\!\!\!\!\!\!\!\!\; {}-{}\log\left( \frac{1}{A+\varepsilon B}\right)\sin(x-y)dy+\int\!\!\!\!\!\!\!\!\!\; {}-{}\log\left(\frac{1}{A+\varepsilon B}\right)  f_1(y)\sin(x-y)dy\\
			&=\frac{1}{\varepsilon}\int\!\!\!\!\!\!\!\!\!\; {}-{} \log\left(\frac{1}{A}\right)\sin(x-y)dy-\frac{1}{\varepsilon}\int\!\!\!\!\!\!\!\!\!\; {}-{}\int_0^1\frac{\varepsilon B}{A+t\varepsilon B}dt\sin(x-y)dy\nonumber\\
			&\quad+\int\!\!\!\!\!\!\!\!\!\; {}-{} \log\left(\frac{1}{A}\right)f(y)\sin(x-y)dy-\int\!\!\!\!\!\!\!\!\!\; {}-{}\int_0^1\frac{\varepsilon B}{A+t\varepsilon B}dtf_1(y)\sin(x-y)dy\nonumber\\
			&=-\int\!\!\!\!\!\!\!\!\!\; {}-{}\int_0^1\frac{B}{A+t\varepsilon B}dt\sin(x-y)dy+\int\!\!\!\!\!\!\!\!\!\; {}-{} \log\left(\frac{1}{A}\right)f_1(y)\sin(x-y)dy\nonumber\\
			&\quad-\varepsilon\int\!\!\!\!\!\!\!\!\!\; {}-{}\int_0^1\frac{ B}{A+t\varepsilon B}dtf(y)\sin(x-y)dy\nonumber\\
			&=-\int\!\!\!\!\!\!\!\!\!\; {}-{}\int_0^1\frac{B}{A}dt\sin(x-y)dy+\int\!\!\!\!\!\!\!\!\!\; {}-{}\int_0^1\int_0^1\frac{t\varepsilon B^2}{(A+t\tau \varepsilon B)^2}d\tau dt\sin(x-y)dy\nonumber\\
			&\quad +\int\!\!\!\!\!\!\!\!\!\; {}-{} \log\left(\frac{1}{A}\right)f(y)\sin(x-y)dy-\varepsilon\int\!\!\!\!\!\!\!\!\!\; {}-{}\int_0^1\frac{ B}{A+t\varepsilon B}dtf_1(y)\sin(x-y)dy\nonumber\\
			&=-\int\!\!\!\!\!\!\!\!\!\; {}-{}f_1(y)\sin(x-y)dy+\int\!\!\!\!\!\!\!\!\!\; {}-{} \log\left(\frac{1}{A}\right)f_1(y)\sin(x-y)dy+\varepsilon \mathcal{R}_{121}(\varepsilon, f_1),\nonumber
		\end{split}
	\end{equation*}
	where $\mathcal{R}_{121}(\varepsilon,f_1)$ is not singular with respect to $\varepsilon$ and belongs to $L^2$ due to $|B|\leq C|x-y|^2$. In view of \eqref{3-2}, it is easy to see $\int\!\!\!\!\!\!\!\!\!\; {}-{} \log\left(\frac{1}{A}\right)f_1(y)\sin(x-y)dy\in L^2$, and hence $G_{12}\in L^2$ for all $\varepsilon\not=0$.
	
	Applying \eqref{3-3} to $\frac{1}{2}\int\!\!\!\!\!\!\!\!\!\; {}-{} \log\left(\frac{1}{ \left(R_1(x)-R_1(y)\right)^2+4R_1(x)R_1(y)\sin^2\left(\frac{x-y}{2}\right)}\right)(f_1'(y)-f_1'(x))\cos(x-y)dy$ we obtain
	\begin{align*}
		&\quad \frac{1}{2}\int\!\!\!\!\!\!\!\!\!\; {}-{} \log\left(\frac{1}{ \left(R_1(x)-R_1(y)\right)^2+4R_1(x)R_1(y)\sin^2\left(\frac{x-y}{2}\right)}\right)(f_1'(y)-f_1'(x))\cos(x-y)dy\\
		&=\frac{1}{2}\int\!\!\!\!\!\!\!\!\!\; {}-{} \log\left(\frac{1}{A+\varepsilon B}\right)(f_1'(y)-f_1'(x))\cos(x-y)dy\\
		&=\frac{1}{2}\int\!\!\!\!\!\!\!\!\!\; {}-{} \log\left(\frac{1}{A}\right)(f_1'(y)-f_1'(x))\cos(x-y)dy\\
		&\quad-\frac{1}{2}\int\!\!\!\!\!\!\!\!\!\; {}-{}\int_0^1\frac{\varepsilon B}{A+t\varepsilon B}dt(f_1'(y)-f_1'(x))\cos(x-y)dy\nonumber\\
		&=\frac{1}{2}\int\!\!\!\!\!\!\!\!\!\; {}-{} \log\left(\frac{1}{A}\right)(f_1'(y)-f_1'(x))\cos(x-y)dy+\varepsilon \mathcal{R}_{122}(\varepsilon,f_1).
	\end{align*}

	 Define
	\begin{align*}
		G_{12}(0, f_1)&=\frac{1}{2}\int\!\!\!\!\!\!\!\!\!\; {}-{} \log\left(\frac{1}{A}\right)\left(f_1(y)\sin(x-y)+(f_1'(y)-f_1'(x))\cos(x-y)\right)dy\\
		&\quad-\frac{1}{2}\int\!\!\!\!\!\!\!\!\!\; {}-{}f_1(y)\sin(x-y)dy\\
		&=\frac{1}{2}\int\!\!\!\!\!\!\!\!\!\; {}-{} \log\left(\frac{1}{4\sin^2\left(\frac{x-y}{2}\right)}\right)\left(f_1(y)\sin(x-y)+(f_1'(y)-f_1'(x))\cos(x-y)\right)dy\\
		&\quad -\frac{1}{2}\int\!\!\!\!\!\!\!\!\!\; {}-{}f_1(y)\sin(x-y)dy\\
		&=\frac{1}{2}\int\!\!\!\!\!\!\!\!\!\; {}-{} \log\left(\frac{1}{\sin^2\left(\frac{y}{2}\right)}\right)\left(f_1(x-y)\sin(y)+(f_1'(x-y)-f_1'(x))\cos(y)\right)dy\\
		&\quad -\frac{1}{2}\int\!\!\!\!\!\!\!\!\!\; {}-{}f_1(y)\sin(x-y)dy.
	\end{align*}   Then $G_{12}\in L^2$ for all $(\varepsilon, \boldsymbol f)\in \left(-\frac{1}{2}, \frac{1}{2}\right)\times  V $.
	
	We are going to show that $\partial^{k-1}G_{12}\in L^2$.
	Splitting the term  $R_1(x)R'_1(y)-R'_1(x)R_1(y)$ into $R_1(x)(R'_1(y)-R'_1(x))+R'_1(x)(R_1(x)-R_1(y))$ and taking $(k-1)$th derivatives of $G_{12}$,  the most singular term is 	
	\begin{equation*}
		\begin{split}
			&-\frac{\partial^{k-1}f_1(x)}{R_1(x)} \varepsilon G_{12}\\
			&+\int\!\!\!\!\!\!\!\!\!\; {}-{}\log\left(\frac{1}{ \left(R_1(x)-R_1(y)\right)^2+4R_1(x)R_1(y)\sin^2\left(\frac{x-y}{2}\right)}\right)\\
			&\qquad\times \left((\partial^k f(x)  f'(y)+f'(x)\partial^k f(y))+(\partial^kf(y)-\partial^kf(x))\cos(x-y)\right)dy\\	
			& -\int\!\!\!\!\!\!\!\!\!\; {}-{}\frac{(f(x)-f(y))(\partial^{k-1}f(x)-\partial^{k-1}f(y))+2(f(x)\partial^{k-1}f(y)+\partial^{k-1}f(x)f(y))\sin^2(\frac{x-y}{2})}{\left(R_1(x)-R_1(y)\right)^2+4R_1(x)R_1(y)\sin^2\left(\frac{x-y}{2}\right)}\\
			& \quad\times \left[(R_1(x)R_1(y)+R'_1(x)R'_1(y))\sin(x-y)+(R_1(x)R'_1(y)-R'_1(x)R_1(y))\cos(x-y)\right] dy\nonumber\\
			& -\int\!\!\!\!\!\!\!\!\!\; {}-{}\frac{(f(x)-f(y))(f'(x)-f'(y))+2(f(x)f'(y)+f'(x)f(y))\sin^2(\frac{x-y}{2})}{\left(R_1(x)-R_1(y)\right)^2+4R_1(x)R_1(y)\sin^2\left(\frac{x-y}{2}\right)}\\
			& \quad\times \left[\varepsilon(\partial^{k-1}f_1(x)R'_1(y)+f_1'(x)\partial^{k-1}f_1(y))\sin(x-y)\right.\\
			&\qquad\left.+(R_1(x)\partial^{k-1}f_1(y)-\partial^{k-1}f_1(x)R_1(y))\cos(x-y)\right] dy.\nonumber
		\end{split}
	\end{equation*}
 Since $f\in H^k$ and $k\ge3$, we have $\|\partial^lf\|_{L^\infty}<C \|f\|_{H^k}$ for $l=0,1,2$. By  mean value theorem, we can deduce that
 \begin{align*}
 	&\quad |(f(x)-f(y))(\partial^{k-1}f(x)-\partial^{k-1}f(y))+2(f(x)\partial^{k-1}f(y)+\partial^{k-1}f(x)f(y))\sin^2(\frac{x-y}{2})|\\
 	&\leq C|x-y|(|\partial^{k-1}f(x)|+|\partial^{k-1}f(y)|),\\
 	&\\
 	&\quad\left|(R_1(x)R_1(y)+R'_1(x)R'_1(y))\sin(x-y)\right|\leq C|x-y|,\\
 	&\\
 	&\quad\left|(R_1(x)R'_1(y)-R'_1(x)R_1(y))\cos(x-y)\right|\leq C|x-y|,\\
 	&\\
 	&\quad \left|(f(x)-f(y))(f'(x)-f'(y))+2(f(x)f'(y)+f'(x)f(y))\sin^2(\frac{x-y}{2})\right|\leq C|x-y|^2.
 \end{align*}
Therefore, by H\"older's inequality, we have $\partial^{k-1} G_{12}\in L^2$, which implies $G_{12}\in H^{k-1}$.  Notice that $f_1(x)\in V$ is an even function, and $f'_1(x)$ is odd. By changing $y$ to $-y$ in $G_{12}$, we can deduce that $G_{12}(\varepsilon,\boldsymbol f)$ is odd, and hence $G_{12}\in Y^{k-1}$.

	Next, we prove the continuity of $G_{12}$. By the definition of $G_{12}(0, \boldsymbol f)$ and above calculations, one can easily check that $G_{12}$ is continuous with respect to $\varepsilon$ at $\varepsilon=0$ and $G_{12}(0,\boldsymbol f)$ is continuous with respect to $\boldsymbol f$. Thus, we only need to prove that $G_{12}$ is continuous for $\varepsilon\not=0$. We deal with the most singular term
	\begin{align*}
		G_{122}&=\frac{1}{2}\int\!\!\!\!\!\!\!\!\!\; {}-{} \log\left(\frac{1}{ \left(R_1(x)-R_1(y)\right)^2+4R_1(x)R_1(y)\sin^2\left(\frac{x-y}{2}\right)}\right)(f'_1(y)-f'_1(x))\cos(x-y) dy\\
		&=\frac{1}{2}\int\!\!\!\!\!\!\!\!\!\; {}-{} \log\left(\frac{1}{ \varepsilon^2\left(f_1(x)-f_1(y)\right)^2+4(1+\varepsilon f_1(x))(1+f_1(y))\sin^2\left(\frac{x-y}{2}\right)}\right)\\
		&\quad \times(f'_1(y)-f'_1(x))\cos(x-y) dy.
	\end{align*}
The following notations will be used for convenience. For a general function $h$, denote
	$$\Delta g=g(x)-g(y), \ \ \ g=g(x), \ \ \ \tilde g=g(y),$$
	and
	$$D(h)=\varepsilon^{2}(\Delta h)^2+4(1+\varepsilon h)(1+\varepsilon\tilde h)\sin^2(\frac{x-y}{2}).$$
	Then for $f_1,f_2\in V$, it holds
		\begin{align*}
			&\quad G_{122}(\varepsilon, f_1)-G_{22}(\varepsilon, f_2)\\
			&=-\frac{1}{2}\int\!\!\!\!\!\!\!\!\!\; {}-{} \log\left(\frac{1}{ D(f_1)}\right)\Delta f_1'\cos(x-y) dy+\frac{1}{2}\int\!\!\!\!\!\!\!\!\!\; {}-{} \log\left(\frac{1}{ D(f_2)}\right)\Delta f_2'\cos(x-y) dy\\
			&=-\frac{1}{2}\int\!\!\!\!\!\!\!\!\!\; {}-{} \log\left(\frac{1}{ D(f_1)}\right)(\Delta f_1'-\Delta f_2')\cos(x-y) dy\\
			&\quad+\frac{1}{2}\int\!\!\!\!\!\!\!\!\!\; {}-{} \left(\log\left(\frac{1}{ D(f_2)}\right)-\log\left(\frac{1}{ D(f_1)}\right)\right)\Delta f_2'\cos(x-y) dy\\
			&=I_1+I_2.
		\end{align*}
	Similar to $G_{12}$, it is easy to show that $\|I_1\|_{H^{k-1}}\le C\|f_1-f_2\|_{H^k}$. To show  $\|I_2\|_{H^{k-1}}\le C\|f_1-f_2\|_{H^k}$, it is enough to prove $$\left|\left|\partial^{l} \left(\log\left(\frac{1}{ D(f_2)}\right)-\log\left(\frac{1}{ D(f_1)}\right)\right)\cdot|x-y|\right|\right|_{L^2}\leq C\|f_1-f_2\|_{H^k},$$
	for $l=0,\ldots, k-1$, which can be verified by direct calculations. Indeed, we consider the most singular term
	\begin{align*}
		&\quad\left(\frac{\partial^{k-1} D(f_1)}{D(f_1)}-\frac{\partial^{k-1} D(f_2)}{D(f_2)}\right)|x-y|\\
		&=\frac{\partial^{k-1} D(f_1)-\partial^{k-1} D(f_2)}{D(f_1)}|x-y|+\left(\frac{ 1}{D(f_1)}-\frac{1}{D(f_2)}\right)|x-y|\partial^{k-1} D(f_2)\\
		&=I_{21}+I_{22}.
	\end{align*}
	Obviously, it holds $\|I_{21}\|_{H^{k-1}}\le C\|f_1-f_2\|_{H^k}$. Using mean value theorem, we have
	\begin{align*}
		I_{22}=&\left(\frac{ 1}{D(f_1)}-\frac{1}{D(f_2)}\right)|x-y|\partial^{k-1} D(f_2)\\
		&=\frac{1}{\xi^2}(D(f_2)-D(f_1))|x-y|\partial^{k-1} D(f_2),
	\end{align*}
	where $\xi$ lies between $D(f_1)$ and $D(f_2)$ and hence $c_1|x-y|^2\leq \xi\leq c_2 |x-y|^2$ for some constant $c_1,c_2>0$. Noticing that $|D(f_2)-D(f_1)|\leq C|x-y|(|f_1(x)-f_2(x)|+|f_1(y)-f_2(y)|)$, it follows that $\|I_{22}\|_{H^{k-1}}\le C\|f_1-f_2\|_{H^k}$.
	Hence we have proved $G_{12}(\varepsilon, \boldsymbol f): \left(-\varepsilon_0, \varepsilon_0\right)\times V\rightarrow Y^{k-1}$ is continuous.
	And $G_{12}$ takes the form
	\begin{equation}\label{3-5}
		\begin{split}
			G_{12}&=\frac{1}{2}\int\!\!\!\!\!\!\!\!\!\; {}-{} \log\left(\frac{1}{\sin^2\left(\frac{y}{2}\right)}\right)\left(f_1(x-y)\sin(y)+(f_1'(x-y)-f_1'(x))\cos(y)\right)dy\\
			&-\frac{1}{2}\int\!\!\!\!\!\!\!\!\!\; {}-{}f_1(y)\sin(x-y)dy+\varepsilon\mathcal{R}_{12}(\varepsilon,\boldsymbol f),
		\end{split}
	\end{equation}
	where $\mathcal{R}_{12}(\varepsilon, \boldsymbol f): \left(-\varepsilon_0, \varepsilon_0 \right)\times V \rightarrow Y^{k-1}$ is continuous by previous discussion.	
	
	Then we turn to consider $G_{13}$. Since $\left(R_1(x)-R_2(y)\right)^2+4R_1(x)R_2(y)\sin^2\left(\frac{x-y}{2}\right)=(1-b)^2+O(\varepsilon)$, the term $G_{13}$ is more regular than $G_{12}$ for $\varepsilon$ small and can be proved to be bounded in $Y^{k-1}$ and continuous as above. Applying the Taylor's formula \eqref{3-3} and \eqref{3-4} to $G_{13}$, one has
	\begin{equation}\label{3-6}
		\begin{split}
		G_{13}&=\frac{\gamma-1}{2}\int\!\!\!\!\!\!\!\!\!\; {}-{} \log\left(\frac{1}{(1-b)^2+4b\sin^2\left(\frac{y}{2}\right)}\right)\left(f_2(x-y)\sin(y)+(f_2'(x-y)-bf_1'(x))\cos(y)\right)dy\\
		&+\frac{1-\gamma}{2}\int\!\!\!\!\!\!\!\!\!\; {}-{}\frac{b\sin(y)f_2(y)\left(-2(1-b)+4\sin^2\left(\frac{y}{2}\right)\right)}{(1-b)^2+4b\sin^2\left(\frac{y}{2}\right)}dy+\varepsilon\mathcal{R}_{13}(\varepsilon,\boldsymbol f),
			\end{split}
	\end{equation}
   where $\mathcal{R}_{13}(\varepsilon, \boldsymbol f): \left(-\varepsilon_0, \varepsilon_0\right)\times V \rightarrow Y^{k-1}$ is continuous.

   Similarly, one has that the denominators of $G_{14}$ and $G_{15}$ are bounded from below by a positive constant  for $\varepsilon$ small. Hence $G_{14}$ and $G_{15}$ are more regular than $G_{12}$ and can be proved to be bounded in $Y^{k-1}$ and continuous more easily than $G_{12}$. Thus, the proof of continuity of $G_{1}^0$  is complete. For later use, we compute the expansion of $G^0_1$ at $\varepsilon=0$.
   Denote
   $$A_n=\left(-1+\cos\left(\frac{2\pi n}{N}\right)\right)^2 d^2+\sin^2\left(\frac{2\pi n}{N}\right)d^2>0,$$
   $$B_n=2d(-1+\cos(\frac{2\pi n}{N}))(\cos(x)-\cos(y+\frac{2\pi n}{N}))+2d\sin(\frac{2\pi n}{N})(\sin(x)-\sin(x+\frac{2\pi n}{N})),$$
   $$\tilde B_n=2d(-1+\cos(\frac{2\pi n}{N}))(\cos(x)-b\cos(y+\frac{2\pi n}{N}))+2d\sin(\frac{2\pi n}{N})(\sin(x)-b\sin(x+\frac{2\pi n}{N})).$$
   Then, we have $$\left|\boldsymbol z_1(x)-d\boldsymbol e_1- Q_{\frac{2\pi n}{N}}\left(\boldsymbol z_1(y)-d\boldsymbol e_1\right)\right|^2=A_n+\varepsilon B_n +O(\varepsilon^2) ,$$
   and
   $$\left|\boldsymbol z_1(x)-d\boldsymbol e_1- Q_{\frac{2\pi n}{N}}\left(\boldsymbol z_2(y)-d\boldsymbol e_1\right)\right|^2=A_n+\varepsilon \tilde B_n +O(\varepsilon^2).$$
   Applying the Taylor's formula \eqref{3-3} and \eqref{3-4}, we have
   \begin{align}\label{3-7}
   	   G_{14}&=\sum_{n=1}^{N-1} \frac{1}{ \varepsilon} \int\!\!\!\!\!\!\!\!\!\; {}-{}\log\left( \frac{1}{\left|\boldsymbol z_1(x)-d\boldsymbol e_1- Q_{\frac{2\pi n}{N}}\left(\boldsymbol z_1(y)-d\boldsymbol e_1\right)\right|}\right)\sin(x-y-\frac{2\pi n}{N}) dy\\
   	    &+\sum_{n=1}^{N-1}  \int\!\!\!\!\!\!\!\!\!\; {}-{}\log\left( \frac{1}{\left| \left(\boldsymbol z_1(x)-d\boldsymbol e_1- Q_{\frac{2\pi n}{N}}\left(\boldsymbol z_1(y)-d\boldsymbol e_1\right)\right)\right|}\right)\nonumber\\
   	    &\quad \times\left( f_1(y)\sin(x-y-\frac{2\pi n}{N})+ (f_1'(y)-f_1'(x))\cos(x-y-\frac{2\pi n}{N})\right)dy+\varepsilon\mathcal{R}_{14}\nonumber  \\
   	    &=\sum_{n=1}^{N-1} \frac{1}{ 2\varepsilon} \int\!\!\!\!\!\!\!\!\!\; {}-{}\log\left( \frac{1}{A_n+\varepsilon B_n+O(\varepsilon^2)}\right)\sin(x-y-\frac{2\pi n}{N}) dy\nonumber\\
   	    &+\sum_{n=1}^{N-1}  \int\!\!\!\!\!\!\!\!\!\; {}-{}\log\left( \frac{1}{A_n}\right) \left( f_1(y)\sin(x-y-\frac{2\pi n}{N})+ (f_1'(y)-f_1'(x))\cos(x-y-\frac{2\pi n}{N})\right)dy\nonumber\\
   	    &\quad+\varepsilon\mathcal{R}_{14}\nonumber  \\
   	    &=\sum_{n=1}^{N-1} \frac{1}{ 2\varepsilon} \int\!\!\!\!\!\!\!\!\!\; {}-{}\log\left( \frac{1}{A_n+\varepsilon B_n+O(\varepsilon^2)}\right)\sin(x-y-\frac{2\pi n}{N}) dy+\varepsilon\mathcal{R}_{14}\nonumber  \\
   	    &=\sum_{n=1}^{N-1} \frac{1}{ 2\varepsilon} \int\!\!\!\!\!\!\!\!\!\; {}-{}\log\left( \frac{1}{A_n}\right)\sin(x-y-\frac{2\pi n}{N}) dy\nonumber\\
   	    &\quad-\sum_{n=1}^{N-1} \frac{1}{ 2} \int\!\!\!\!\!\!\!\!\!\; {}-{} \frac{B_n}{A_n+\varepsilon tB_n+tO(\varepsilon^2)}\sin(x-y-\frac{2\pi n}{N}) dy+\varepsilon\mathcal{R}_{14}\nonumber  \\	
   	    &=-\sum_{n=1}^{N-1} \frac{1}{ 2} \int\!\!\!\!\!\!\!\!\!\; {}-{} \frac{B_n}{A_n}\sin(x-y-\frac{2\pi n}{N}) dy+\varepsilon\mathcal{R}_{14}\nonumber  \\
   	    &=-\sum_{n=1}^{N-1} \frac{(1-\cos(\frac{2\pi n}{N}))\sin (x)}{2d\left((-1+\cos(\frac{2\pi n}{N}))^2 +\sin^2(\frac{2\pi n}{N})\right)}+\varepsilon\mathcal{R}_{14}(\varepsilon, \boldsymbol f),\nonumber
   \end{align}
where $\mathcal{R}_{14}(\varepsilon, \boldsymbol f)$ is regular and may be different from line to line.

 Applying the Taylor's formula \eqref{3-3} and \eqref{3-4}, we obtain
\begin{align}\label{3-8}
{}&G_{15}=\\
	 &\ \sum_{n=1}^{N-1} \frac{\gamma-1}{ \varepsilon} \int\!\!\!\!\!\!\!\!\!\; {}-{}\log\left( \frac{1}{\left| \boldsymbol z_1(x)-d\boldsymbol e_1- Q_{\frac{2\pi n}{N}}\left(\boldsymbol z_2(y)-d\boldsymbol e_1\right)\right|}\right) b\sin(x-y-\frac{2\pi n}{N}) dy\\
	&+\sum_{n=1}^{N-1}  (\gamma-1)\int\!\!\!\!\!\!\!\!\!\; {}-{}\log\left( \frac{1}{\left| \boldsymbol z_1(x)-d\boldsymbol e_1- Q_{\frac{2\pi n}{N}}\left(\boldsymbol z_2(y)-d\boldsymbol e_1\right)\right|}\right)\nonumber\\
	&\quad \times\left( f_2(y)\sin(x-y-\frac{2\pi n}{N})+ (f_2'(y)-bf_1'(x))\cos(x-y-\frac{2\pi n}{N})\right)dy+\varepsilon\mathcal{R}_{15}\nonumber  \\
	&=\sum_{n=1}^{N-1} \frac{(\gamma-1)b}{ 2\varepsilon} \int\!\!\!\!\!\!\!\!\!\; {}-{}\log\left( \frac{1}{A_n+\varepsilon \tilde{B}_n+O(\varepsilon^2)}\right)\sin(x-y-\frac{2\pi n}{N}) dy\nonumber\\
	&+\sum_{n=1}^{N-1}  (\gamma-1)\int\!\!\!\!\!\!\!\!\!\; {}-{}\log\left( \frac{1}{A_n}\right) \left( f_2(y)\sin(x-y-\frac{2\pi n}{N})+ (f_2'(y)-f_1'(x))\cos(x-y-\frac{2\pi n}{N})\right)dy\nonumber\\
	&\quad+\varepsilon\mathcal{R}_{15}\nonumber  \\
	&=\sum_{n=1}^{N-1} \frac{b}{ 2\varepsilon} \int\!\!\!\!\!\!\!\!\!\; {}-{}\log\left( \frac{1}{A_n+\varepsilon \tilde{B}_n+O(\varepsilon^2)}\right)\sin(x-y-\frac{2\pi n}{N}) dy+\varepsilon\mathcal{R}_{15}\nonumber  \\
	&=\sum_{n=1}^{N-1} \frac{(\gamma-1)b}{ 2\varepsilon} \int\!\!\!\!\!\!\!\!\!\; {}-{}\log\left( \frac{1}{A_n}\right)\sin(x-y-\frac{2\pi n}{N}) dy\nonumber\\
	&-\sum_{n=1}^{N-1} \frac{(\gamma-1)b}{ 2} \int\!\!\!\!\!\!\!\!\!\; {}-{} \frac{\tilde{B}_n}{A_n+\varepsilon t\tilde{B}_n+tO(\varepsilon^2)}\sin(x-y-\frac{2\pi n}{N}) dy+\varepsilon\mathcal{R}_{15}\nonumber  \\	
	&=-\sum_{n=1}^{N-1} \frac{(\gamma-1)b}{ 2} \int\!\!\!\!\!\!\!\!\!\; {}-{} \frac{\tilde{B}_n}{A_n}\sin(x-y-\frac{2\pi n}{N}) dy+\varepsilon\mathcal{R}_{15}\nonumber  \\
	&=-\sum_{n=1}^{N-1} \frac{(\gamma-1)b^2(1-\cos(\frac{2\pi n}{N}))\sin (x)}{2d\left((-1+\cos(\frac{2\pi n}{N}))^2 +\sin^2(\frac{2\pi n}{N})\right)}+\varepsilon\mathcal{R}_{15}(\varepsilon, \boldsymbol f),\nonumber
\end{align}
where $\mathcal{R}_{15}(\varepsilon,\Om, \boldsymbol f): \left(-\varepsilon_0, \varepsilon_0\right)\times \mathbb{R}\times V \rightarrow Y^{k-1}$ is continuous and may be different from line to line.

Combining \eqref{3-1}, \eqref{3-5}, \eqref{3-6}, \eqref{3-7} and \eqref{3-8} we conclude
\begin{align}\label{3-9}
G^0_1	&=\pi(1-b^2+\gamma b^2)\Omega d\sin(x)\\
	&+\frac{1}{2}\int\!\!\!\!\!\!\!\!\!\; {}-{} \log\left(\frac{1}{\sin^2\left(\frac{y}{2}\right)}\right)\left(f_1(x-y)\sin(y)+(f_1'(x-y)-f_1'(x))\cos(y)\right)dy\nonumber\\
	&-\frac{1}{2}\int\!\!\!\!\!\!\!\!\!\; {}-{}f_1(y)\sin(x-y)dy\nonumber\\
	&\frac{\gamma-1}{2}\int\!\!\!\!\!\!\!\!\!\; {}-{} \log\left(\frac{1}{(1-b)^2+4b\sin^2\left(\frac{y}{2}\right)}\right)\left(f_2(x-y)\sin(y)+(f_2'(x-y)-bf_1'(x))\cos(y)\right)dy\nonumber\\
	&+\frac{1-\gamma}{2}\int\!\!\!\!\!\!\!\!\!\; {}-{}\frac{b\sin(y)f_2(y)\left(-2(1-b)+4\sin^2\left(\frac{y}{2}\right)\right)}{(1-b)^2+4b\sin^2\left(\frac{y}{2}\right)}\nonumber\\
	&-\sum_{n=1}^{N-1} \frac{(1-b^2+\gamma b^2)(1-\cos(\frac{2\pi n}{N}))\sin (x)}{2d\left((-1+\cos(\frac{2\pi n}{N}))^2 +\sin^2(\frac{2\pi n}{N})\right)}+\varepsilon\mathcal{R}_{1}(\varepsilon,\Om, \boldsymbol f),\nonumber
\end{align}
  where $\mathcal{R}_{1}(\varepsilon, \Om, \boldsymbol f): \left(-\varepsilon_0, \varepsilon_0\right)\times \mathbb{R} \times V \rightarrow Y^{k-1}$ is continuous.

  Similarly, we can prove that  	$G_2^0(\varepsilon, \Omega, \boldsymbol f): \left(-\varepsilon_0, \varepsilon_0\right)\times \mathbb{R} \times V \rightarrow Y^{k-1}$ is continuous and is of the form
  \begin{align}\label{3-10}
  G^0_2	& =\pi(1-b^2+\gamma b^2)\Omega d\sin(x)\\
  	&+\frac{1}{2b}\int\!\!\!\!\!\!\!\!\!\; {}-{} \log\left(\frac{1}{(1-b)^2+4b\sin^2\left(\frac{y}{2}\right)}\right)\left(bf_1(x-y)\sin(y)+(bf_1'(x-y)-f_2'(x))\cos(y)\right)dy\nonumber\\
  	&-\frac{1}{2}\int\!\!\!\!\!\!\!\!\!\; {}-{}\frac{\sin(y)f_1(y)\left(2(1-b)+4b\sin^2\left(\frac{y}{2}\right)\right)}{(1-b)^2+4b\sin^2\left(\frac{y}{2}\right)}\nonumber\\
  	&+\frac{\gamma-1}{2}\int\!\!\!\!\!\!\!\!\!\; {}-{} \log\left(\frac{1}{\sin^2\left(\frac{y}{2}\right)}\right)\left(f_2(x-y)\sin(y)+(f_2'(x-y)-f_2'(x))\cos(y)\right)dy\nonumber\\
  	&+\frac{1-\gamma}{2}\int\!\!\!\!\!\!\!\!\!\; {}-{}f_2(y)\sin(x-y)dy\nonumber\\	
  	&-\sum_{n=1}^{N-1} \frac{(1-b^2-\gamma b^2)(1-\cos(\frac{2\pi n}{N}))\sin (x)}{2d\left((-1+\cos(\frac{2\pi n}{N}))^2 +\sin^2(\frac{2\pi n}{N})\right)}+\varepsilon\mathcal{R}_{2}(\varepsilon,\Om, \boldsymbol f),\nonumber
  \end{align}
 where $\mathcal{R}_{2}(\varepsilon, \Om, \boldsymbol f): \left(-\varepsilon_0, \varepsilon_0\right)\times \mathbb{R} \times V \rightarrow Y^{k-1}$ is continuous.

  The proof of Lemma \ref{lm3-1} is thus complete.
\end{proof}

For $(\varepsilon,\Omega, \boldsymbol f)\in \left(-\varepsilon_0, \varepsilon_0\right)\times \mathbb{R} \times V$ and $h\in X^{k}$, denote $\partial_{f_i} G_j^0(\varepsilon, \Omega, f)h$ to be the  Gateaux derivatives.
We have following lemma:
\begin{lemma}\label{lm3-2}
	The Gautex derivative
	$\partial_{f_i} G_j^0(\varepsilon, \Omega, \boldsymbol  f): \left(-\varepsilon_0, \varepsilon_0\right)\times \mathbb{R} \times V \rightarrow L(X^{k}, Y^{k-1})$ exist and is continuous for $i, j=1,2$.	
\end{lemma}
\begin{proof}
		Set
		\begin{align*}
			B_{ijn}:&=2\left(-1+\cos\left(\frac{2\pi n}{N}\right)\right)d \left(R_i(x)\cos(x)-R_j(y)\cos\left(y+\frac{2\pi n}{N}\right)\right)\\
			&+2\sin\left(\frac{2\pi n}{N}\right)\left(R_i(x)\sin(x)-R_j(y)\sin\left(y+\frac{2\pi n}{N}\right)\right)\\
			&+\varepsilon\left(R_i(x)^2+R_j^2+R_i(x)R_j(y)\cos\left(x-y-\frac{2\pi n}{N}\right)\right).
		\end{align*}
	Then
		$$\left| z_i(x)-d\boldsymbol e_1- Q_{\frac{2\pi n}{N}} \left(z_j(y)-d\boldsymbol e_1\right)\right|^2=A_n+\varepsilon B_{ijn}.$$
	
	We are going to prove that the derivative of $G_1^0$ with respect to $f_1$ is as follows,
	\begin{align*}
		\partial_{f_1}G_1^0 (\varepsilon, \Omega, \boldsymbol f) h=F_1+F_2+F_3+F_4+F_5,
	\end{align*}
where $F_1,F_2, F_3, F_4, F_5$ are given by
\begin{align*}
	F_1=\pi(1-b^2+\gamma b^2)\Omega\left( \varepsilon^{2} h'(x)-\frac{d\varepsilon h'(x)\cos(x)}{R_1(x)}+\frac{d\varepsilon R'_1(x)h'(x)\cos(x)}{R_1(x)^2}\right),
\end{align*}
\begin{align*}
{}&F_2=\\
	&\ -\frac{ h(x)}{2R_1(x)^2}\int\!\!\!\!\!\!\!\!\!\; {}-{} \log\left(\frac{1}{ \left(R_1(x)-R_1(y)\right)^2+4R_1(x)R_1(y)\sin^2\left(\frac{x-y}{2}\right)}\right)\nonumber\\
	&\times \left[(R_1(x)R_1(y)+R'_1(x)R'_1(y))\sin(x-y)+(R_1(x)R'_1(y)-R'_1(x)R_1(y))\cos(x-y)\right] dy\nonumber\\
	&+\frac{1}{2R_1(x)}\int\!\!\!\!\!\!\!\!\!\; {}-{}\log\left(\frac{1}{ \left(R_1(x)-R_1(y)\right)^2+4R_1(x)R_1(y)\sin^2\left(\frac{x-y}{2}\right)}\right)\nonumber\\
	&\times \left[(h(x)R_1(y)+R_1(x)h(y)+h'(x)R'_1(y)+R'_1(x)h'(y))\sin(x-y)\right.\nonumber\\
	&+\left.(h(x)R'_1(y)+R_1(x)h'(y)-h'(x)R_1(y)-R'_1(x)h(y))\cos(x-y)\right] dy\nonumber\\
	&-\frac{1}{2R_1(x)}\int\!\!\!\!\!\!\!\!\!\; {}-{} \frac{(R_1(x)R_1(y)+R'_1(x)R'_1(y))\sin(x-y)+(R_1(x)R'_1(y)-R'_1(x)R_1(y))\cos(x-y)}{ \left(R_1(x)-R_1(y)\right)^2+4R_1(x)R_1(y)\sin^2\left(\frac{x-y}{2}\right)}\nonumber\\
	& \times \left[2\left(R_1(x)-R_1(y)\right)(h(x)-h(y))+4(h(x)R_1(y)+R_1(x)h(y))\sin^2\left(\frac{x-y}{2}\right)\right] dy,
\end{align*}
\begin{align*}
	F_3&=	 \frac{(1-\gamma)h(x)}{ 2R_1(x)^2}\int\!\!\!\!\!\!\!\!\!\; {}-{}\log\left(\frac{1}{ \left(R_1(x)-R_2(y)\right)^2+4R_1(x)R_2(y)\sin^2\left(\frac{x-y}{2}\right)}\right)\nonumber\\
	&\quad\times \left[(R_1(x)R_2(y)+R'_1(x)R'_2(y))\sin(x-y)+(R_1(x)R'_2(y)-R'_1(x)R_2(y))\cos(x-y)\right] dy\nonumber\\
	& +\frac{\gamma-1}{ 2R_1(x)}\int\!\!\!\!\!\!\!\!\!\; {}-{}\log\left(\frac{1}{ \left(R_1(x)-R_2(y)\right)^2+4R_1(x)R_2(y)\sin^2\left(\frac{x-y}{2}\right)}\right)\nonumber\\
	&\quad\times
	\left[(h(x)R_2(y)+h'(x)R'_2(y))\sin(x-y)+(h(x)R'_2(y)-h'(x)R_2(y))\cos(x-y)\right]dy\nonumber\\
	&+\frac{1-\gamma}{2R_1(x)}\int\!\!\!\!\!\!\!\!\!\; {}-{}\frac{[R_1(x)R_2(y)+R'_1(x)R'_2(y)]\sin(x-y)+[R_1(x)R'_2(y)-R'_1(x)R_2(y)]\cos(x-y)}{ \left(R_1(x)-R_2(y)\right)^2+4R_1(x)R_2(y)\sin^2\left(\frac{x-y}{2}\right)}\nonumber\\
	&\quad \times \left[2\left(R_1(x)-R_2(y)\right)h(x)+4h(x)R_2(y)\sin^2\left(\frac{x-y}{2}\right)\right] dy,
\end{align*}
\begin{align*}
{}&F_4=\\
	&-\sum_{n=1}^{N-1} \frac{h(x)}{ R_1(x)^2} \int\!\!\!\!\!\!\!\!\!\; {}-{}\log\left( \frac{1}{\left|\boldsymbol z_1(x)-d\boldsymbol e_1- Q_{\frac{2\pi n}{N}} \left(\boldsymbol z_1(y)-d\boldsymbol e_1\right)\right|}\right)\\
	&\quad \times\left[(R_1(x)R_1(y)+R_1'(x)R_1'(y))\sin(x-y-\frac{2\pi n}{N})\right.\nonumber\\
	&\qquad \left.+(R_1(x)R_1'(y)-R_1'(x)R_1(y))\cos(x-y-\frac{2\pi n}{N})\right] dy\nonumber\\
	&+\sum_{n=1}^{N-1} \frac{1}{ R_1(x)} \int\!\!\!\!\!\!\!\!\!\; {}-{}\log\left( \frac{1}{\left| \boldsymbol z_1(x)-d\boldsymbol e_1- Q_{\frac{2\pi n}{N}} \left(\boldsymbol z_1(y)-d\boldsymbol e_1\right)\right|}\right)\nonumber\\
	&\quad \times \left[(h(x)R_1(y)+R_1(x)h(y)+h'(x)R'_1(y)+R'_1(x)h'(y))\sin(x-y-\frac{2\pi n}{N})\right.\nonumber\\
	&\quad\quad+\left.(h(x)R'_1(y)+R_1(x)h'(y)-h'(x)R_1(y)-R'_1(x)h(y))\cos(x-y-\frac{2\pi n}{N})\right] dy\nonumber\\
	&-\sum_{n=1}^{N-1} \frac{ 1}{ 2 R_1(x)} \int\!\!\!\!\!\!\!\!\!\; {}-{} \frac{\partial_{f_1}B_{11n}h}{\left| \boldsymbol z_1(x)-d\boldsymbol e_1- Q_{\frac{2\pi n}{N}} \left(\boldsymbol z_1(y)-d\boldsymbol e_1\right)\right|^2}\left[(R_1(x)R_1(y)+R_1'(x)R_1'(y))\sin(x-y-\frac{2\pi n}{N})\right.\nonumber\\
	&\quad \left. +(R_1(x)R_1'(y)-R_1'(x)R_1(y))\cos(x-y-\frac{2\pi n}{N})\right] dy\nonumber
\end{align*}
and
\begin{align*}
	F_5&=\sum_{n=1}^{N-1} \frac{(1-\gamma)h(x)}{ R_1(x)^2} \int\!\!\!\!\!\!\!\!\!\; {}-{}\log\left( \frac{1}{\left| \boldsymbol z_1(x)-d\boldsymbol e_1- Q_{\frac{2\pi n}{N}} \left(\boldsymbol z_2(y)-d\boldsymbol e_1\right)\right|}\right)\nonumber\\
	&\quad\left[(R_1(x)R_2(y)+R_1'(x)R_2'(y))\sin(x-y-\frac{2\pi n}{N})\right.\nonumber\\
	&\qquad \left.+(R_1(x)R_2'(y)-R_1'(x)R_2(y))\cos(x-y-\frac{2\pi n}{N})\right] dy\nonumber\\
	&+\sum_{n=1}^{N-1} \frac{\gamma-1}{ R_1(x)} \int\!\!\!\!\!\!\!\!\!\; {}-{}\log\left( \frac{1}{\left| \boldsymbol z_1(x)-d\boldsymbol e_1- Q_{\frac{2\pi n}{N}} \left(\boldsymbol z_2(y)-d\boldsymbol e_1\right)\right|}\right)\nonumber\\
	&\quad\times \left[(h(x)R_2(y)+h'(x)R_2'(y))\sin(x-y-\frac{2\pi n}{N})+(h(x)R_2'(y)-h'(x)R_2(y))\cos(x-y-\frac{2\pi n}{N})\right] dy\nonumber\\
	&+\sum_{n=1}^{N-1} \frac{1-\gamma}{ 2R_1(x)} \int\!\!\!\!\!\!\!\!\!\; {}-{} \frac{\partial_{f_1}B_{12n} h}{\left| \boldsymbol z_1(x)-d\boldsymbol e_1- Q_{\frac{2\pi n}{N}} \left(\boldsymbol z_2(y)-d\boldsymbol e_1\right)\right|^2}\left[(R_1(x)R_2(y)+R_1'(x)R_2'(y))\sin(x-y-\frac{2\pi n}{N})\right.\nonumber\\
	&\quad\left. +(R_1(x)R_2'(y)-R_1'(x)R_2(y))\cos(x-y-\frac{2\pi n}{N})\right] dy.
\end{align*}	
	
	To this end, we need to show for $l=1,2,3,4,5$,
	$$\lim\limits_{t\to0}\left\|\frac{G_{1l}^0(\varepsilon,\Om, f_1+th,f_2)-G_{1l}^0(\varepsilon, \Om, f_1,f_2)}{t}-F_l (\varepsilon, \Omega, f)(\varepsilon, \boldsymbol f,h)\right\|_{Y^{k-1}}=0.$$
	
	The derivative of $G_{11}^0$ is obvious. Next, we  only consider the most singular term
	\begin{align*}
		G_{12}=&\frac{1}{2 R_1(x)\varepsilon}\int\!\!\!\!\!\!\!\!\!\; {}-{} \log\left(\frac{1}{ \left(R_1(x)-R_1(y)\right)^2+4R_1(x)R_1(y)\sin^2\left(\frac{x-y}{2}\right)}\right)\nonumber\\
		&\quad\times \left[(R_1(x)R_1(y)+R'_1(x)R'_1(y))\sin(x-y)+(R_1(x)R'_1(y)-R'_1(x)R_1(y))\cos(x-y)\right] dy\nonumber\\
		&=\frac{1}{2 R_1(x)\varepsilon}\int\!\!\!\!\!\!\!\!\!\; {}-{} \log\left(\frac{1}{ \left(R_1(x)-R_1(y)\right)^2+4R_1(x)R_1(y)\sin^2\left(\frac{x-y}{2}\right)}\right)\\
		&\quad \times(R_1(x)R_1(y)+R'_1(x)R'_1(y))\sin(x-y)dy\\
		&+\frac{1}{2}\int\!\!\!\!\!\!\!\!\!\; {}-{} \log\left(\frac{1}{ \left(R_1(x)-R_1(y)\right)^2+4R_1(x)R_1(y)\sin^2\left(\frac{x-y}{2}\right)}\right)(f'_1(y)-f'_1(x))\cos(x-y)dy\\
		&+\frac{f_1'(x)}{2R_1(x)}\int\!\!\!\!\!\!\!\!\!\; {}-{} \log\left(\frac{1}{ \left(R_1(x)-R_1(y)\right)^2+4R_1(x)R_1(y)\sin^2\left(\frac{x-y}{2}\right)}\right)(f_1(x)-f_1(y))\cos(x-y)dy\\
		&=G_{121}+G_{122}+G_{123}.
	\end{align*}
The proof for other terms is much easier, so we omit it.

     Recall that $R_1=1+\varepsilon f_1$, thus the $\varepsilon$ in the denominator of $G_{12}$ will be eliminated after taking derivatives. Thus, the $\varepsilon$ in $G_{12}$ causes no singularity. For simplicity, we consider the most singular term in $G_{12}$
	\begin{align*}
		G_{122}=\frac{1}{2}\int\!\!\!\!\!\!\!\!\!\; {}-{} \log\left(\frac{1}{ \left(R_1(x)-R_1(y)\right)^2+4R_1(x)R_1(y)\sin^2\left(\frac{x-y}{2}\right)}\right)(f'_1(y)-f'_1(x))\cos(x-y) dy.\nonumber\\
	\end{align*}
   We need to prove
   \begin{equation}\label{3-11}
   	\lim\limits_{t\to0}\left\|\frac{G_{122}(\varepsilon, f_1+th)-G_{122}(\varepsilon, f_1)}{t}-F_{22} (\varepsilon, f)(\varepsilon, f,h)\right\|_{Y^{k-1}}= 0,
   \end{equation}
   where $F_{22}$ is
   \begin{align*}
   	F_{22}&=\frac{1}{2}\int\!\!\!\!\!\!\!\!\!\; {}-{}\log\left(\frac{1}{ \left(R_1(x)-R_1(y)\right)^2+4R_1(x)R_1(y)\sin^2\left(\frac{x-y}{2}\right)}\right) (h'(y)-h'(x))\cos(x-y)dy\nonumber\\
   	&-\frac{1}{2}\int\!\!\!\!\!\!\!\!\!\; {}-{} \frac{(f'_1(y)-f'_1(x))\cos(x-y)}{ \left(R_1(x)-R_1(y)\right)^2+4R_1(x)R_1(y)\sin^2\left(\frac{x-y}{2}\right)}\nonumber\\
   	&\quad \times \left[2\varepsilon\left(R_1(x)-R_1(y)\right)(h(x)-h(y))+4\varepsilon(h(x)R_1(y)+R_1(x)h(y))\sin^2\left(\frac{x-y}{2}\right)\right] dy.
   \end{align*}

	Using the notations given in Lemma \ref{lm3-1}, we deduce
	\begin{equation*}
		\begin{split}
			&\frac{G_{122}(\varepsilon, f_1+th)-G_{122}(\varepsilon, f_1)}{t}-F_{22}(\varepsilon, f_1,h)\\
			&=\frac{1}{t}\int\!\!\!\!\!\!\!\!\!\; {}-{}(f_1'(x)-f_1'(y))\cos(x-y)\\
			&\quad \times\left(\log\left(\frac{1}{D(f_1+th)}\right)-\log\left(\frac{1}{D(f_1)}\right)+t\frac{2\varepsilon^2\Delta f_1\Delta h+4\varepsilon(\tilde R_1h+\tilde hR_1)\sin^2(\frac{x-y}{2})}{D(f_1)}\right)dy\\
			& +\int\!\!\!\!\!\!\!\!\!\; {}-{}(h'(x)-h'(y))\cos(x-y)\left(\log\left(\frac{1}{D(f_1+th)}\right)-\left(\frac{1}{D(f_1)}\right)\right)dy\\
			&=F_{221}+F_{222}.
		\end{split}
	\end{equation*}
	By taking ${k-1}$th partial derivatives of $F_{221}$, we find the most singular term is
	\begin{equation*}
		\begin{split}
			&\quad \frac{1}{t}\int\!\!\!\!\!\!\!\!\!\; {}-{}\left(\log\left(\frac{1}{D(f_1+th)}\right)-\log\left(\frac{1}{D(f_1)}\right)+t\frac{2\varepsilon^2\Delta f_1\Delta h+4\varepsilon(\tilde R_1h+\tilde hR_1)\sin^2(\frac{x-y}{2})}{D(f_1)}\right)\\
			& \times(\partial^kf(x)-\partial^kf(y))\cos(x-y)dy.
		\end{split}
	\end{equation*}
	Using the  mean value theorem, we have
	\begin{equation*}
		\left|\log\left(\frac{1}{D(f_1+th)}\right)-\log\left(\frac{1}{D(f_1)}\right)+t\frac{2\varepsilon^2\Delta f_1\Delta h+4\varepsilon(\tilde R_1h+\tilde hR_1)\sin^2(\frac{x-y}{2})}{D_1(f)}\right|\leq  Ct^2\zeta(\varepsilon,f,h),
	\end{equation*}
where $\|\zeta(\varepsilon,f,h)\|_{L^\infty}<\infty$. It follows that
	\begin{equation*}
		\|F_{221}\|_{Y^{k-1}}\le Ct^2\left\|\int\!\!\!\!\!\!\!\!\!\; {}-{}\partial^kf(x)-\partial^kf(y)dy\right\|_{L^2}\le Ct^2\|f\|_{X^k}.
	\end{equation*}
	Taking ${k-1}$th partial derivatives of $F_{222}$, the most singular term is
	\begin{align*}
		\int\!\!\!\!\!\!\!\!\!\; {}-{}(\partial^kh(x)-\partial^kh(y))\cos(x-y)\left(\log\left(\frac{1}{D(f_1+th)}\right)-\left(\frac{1}{D(f_1)}\right)\right)dy.
	\end{align*}
     Using the  mean value theorem, we have
     \begin{align*}
     	&\log\left(\frac{1}{D(f_1+th)}\right)-\left(\frac{1}{D(f_1)}\right)\\
     	&=-\frac{1}{\xi(\varepsilon,f_1,h)}\left(D(f_1+th)-D(f_1)\right)\\
     	&=-\frac{t}{\xi(\varepsilon,f_1,h)}\left(2\varepsilon^2\Delta f_1\Delta h +t(\Delta h)^2+4\varepsilon (h\tilde{R_1}+R_1\tilde{h}\sin^2\left(\frac{x-y}{2}\right) \right),
     \end{align*}
   where $\xi(\varepsilon,f_1,h)$ lies between $D(f_1+th)$ and $D(f_1)$ and hence $0<c_1|x-y|^2\leq |\xi(\varepsilon,f_1,h)|\leq c_2|x-y|^2$ for some positive constants $c_1$ and $c_2$. Noticing that $$\left|2\varepsilon^2\Delta f_1\Delta h +t(\Delta h)^2+4\varepsilon (h\tilde{R_1}+R_1\tilde{h}\sin^2\left(\frac{x-y}{2}\right)\right|\leq C|x-y|^2,$$
   we derive $$\left|\log\left(\frac{1}{D(f_1+th)}\right)-\left(\frac{1}{D(f_1)}\right)\right|\leq Ct,$$
   which implies
   \begin{equation*}
   	\|F_{222}\|_{Y^{k-1}}\le Ct\left\|\int\!\!\!\!\!\!\!\!\!\; {}-{}\partial^kh(x)-\partial^kh(y)dy\right\|_{L^2}\le Ct\|h\|_{X^k}.
   \end{equation*}
   Letting $t\rightarrow 0$, we obtain \eqref{3-11}. And hence, we obtain the existence of Gateaux derivative of $G_{22}$. To prove the continuity of $\partial_{f_1} G_{22}(\varepsilon, f)h$, one just need to verify by the definition. Since there is no other new idea than the proof of continuity for $G_1^0(\varepsilon, \Omega, f)$, we omit it.
   Similarly, one can prove the existence and continuity of all derivatives $\partial_{f_i} G_j^0$, we omit the details and give them directly.
   \begin{align*}
   	&\quad\partial_{f_2} G_1^0 (\varepsilon, \Omega, \boldsymbol{f}) h\\
   	&=\frac{\gamma-1}{ 2 R_1(x)}\int\!\!\!\!\!\!\!\!\!\; {}-{}\log\left(\frac{1}{ \left(R_1(x)-R_2(y)\right)^2+4R_1(x)R_2(y)\sin^2\left(\frac{x-y}{2}\right)}\right)\nonumber\\
   	&\quad\times \left[(R_1(x)h(y)+R'_1(x)h'(y))\sin(x-y)+(R_1(x)h'(y)-R'_1(x)h(y))\cos(x-y)\right]dy\nonumber\\
   	&+\frac{1-\gamma}{ 2 R_1(x)}\int\!\!\!\!\!\!\!\!\!\; {}-{}\frac{-2\left(R_1(x)-R_2(y)\right)h(y)+4R_1(x)h(y)\sin^2\left(\frac{x-y}{2}\right)}{ \left(R_1(x)-R_2(y)\right)^2+4R_1(x)R_2(y)\sin^2\left(\frac{x-y}{2}\right)}\nonumber\\
   	&\quad\times \left[(R_1(x)R_2(y)+R'_1(x)R'_2(y))\sin(x-y)+(R_1(x)R'_2(y)-R'_1(x)R_2(y))\cos(x-y)\right]dy\nonumber\\
   	&+\sum_{n=1}^{N-1} \frac{\gamma-1}{ R_1(x)} \int\!\!\!\!\!\!\!\!\!\; {}-{}\log\left( \frac{1}{\left| \boldsymbol z_1(x)-d\boldsymbol e_1- Q_{\frac{2\pi n}{N}} \left(\boldsymbol z_2(y)-d\boldsymbol e_1\right)\right|}\right)\nonumber\\
   	&\quad\,\times \left[(R_1(x)h(y)+R_1'(x)h'(y))\sin(x-y-\frac{2\pi n}{N})\right.\\
   	&\qquad\left.+(R_1(x)h'(y)-R_1'(x)h(y))\cos(x-y-\frac{2\pi n}{N})\right] dy\nonumber\\
   	&+\sum_{n=1}^{N-1} \frac{\gamma-1}{ 2R_1(x)} \int\!\!\!\!\!\!\!\!\!\; {}-{} \frac{\partial_{f_2}B_{12n}h}{\left| \boldsymbol z_1(x)-d\boldsymbol e_1- Q_{\frac{2\pi n}{N}} \left(\boldsymbol z_2(y)-d\boldsymbol e_1\right)\right|^2}\nonumber\\
   	&\quad\,\times \left[(R_1(x)R_2(y)+R_1'(x)R_2'(y))\sin(x-y-\frac{2\pi n}{N})\right.\\\
   	&\qquad\left.+(R_1(x)R_2'(y)-R_1'(x)R_2(y))\cos(x-y-\frac{2\pi n}{N})\right] dy,\nonumber\\
   \end{align*}

   \begin{align*}
   	&\quad \partial_{f_1} G_2^0 (\varepsilon, \Omega, \boldsymbol{f}) h\\
   	&=\frac{1}{2 R_2(x)}\int\!\!\!\!\!\!\!\!\!\; {}-{} \log\left(\frac{1}{ \left(R_2(x)-R_1(y)\right)^2+4R_2(x)R_1(y)\sin^2\left(\frac{x-y}{2}\right)}\right)\nonumber\\
   	&\quad\times \left[(R_2(x)h(y)+R'_2(x)h'(y))\sin(x-y)+(R_2(x)h'(y)-R'_2(x)h(y))\cos(x-y)\right] dy\nonumber\\
   	&-\frac{1}{2 R_2(x)\varepsilon}\int\!\!\!\!\!\!\!\!\!\; {}-{} \frac{-2\left(R_2(x)-R_1(y)\right)h(y)+4R_2(x)h(y)\sin^2\left(\frac{x-y}{2}\right)}{ \left(R_2(x)-R_1(y)\right)^2+4R_2(x)R_1(y)\sin^2\left(\frac{x-y}{2}\right)}\nonumber\\
   	&\quad\times \left[(R_2(x)R_1(y)+R'_2(x)R'_1(y))\sin(x-y)+(R_2(x)R'_1(y)-R'_2(x)R_1(y))\cos(x-y)\right] dy\nonumber\\
   	&+\sum_{n=1}^{N-1} \frac{1}{R_2(x)} \int\!\!\!\!\!\!\!\!\!\; {}-{}\log\left( \frac{1}{\left| \boldsymbol z_2(x)-d\boldsymbol e_1- Q_{\frac{2\pi n}{N}} \left(\boldsymbol z_1(y)-d\boldsymbol e_1\right)\right|}\right)\nonumber\\
   	&\quad \times \left[(R_2(x)h(y)+R_2'(x)h'(y))\sin(x-y-\frac{2\pi n}{N})\right.\\
   	&\qquad \left.+(R_2(x)h'(y)-R_2'(x)h(y))\cos(x-y-\frac{2\pi n}{N})\right] dy\nonumber\\
   	&-\sum_{n=1}^{N-1} \frac{1}{ R_2(x) } \int\!\!\!\!\!\!\!\!\!\; {}-{}\frac{\partial_{f_1}B_{21n}h}{\left| \boldsymbol z_2(x)-d\boldsymbol e_1- Q_{\frac{2\pi n}{N}} \left(\boldsymbol z_1(y)-d\boldsymbol e_1\right)\right|}\nonumber\\
   	&\quad \times \left[(R_2(x)R_1(y)+R_2'(x)R_1'(y))\sin(x-y-\frac{2\pi n}{N})\right.\\
   	&\qquad\left.+(R_2(x)R_1'(y)-R_2'(x)R_1(y))\cos(x-y-\frac{2\pi n}{N})\right] dy,\nonumber\\
   \end{align*}

   \begin{align*}
   	&\quad\partial_{f_2}G_2^0 (\varepsilon, \Omega, \boldsymbol f) h\\
   	&=(1-b^2-\gamma b^2)\Omega\left( \varepsilon^{2} h'(x)-\frac{d\varepsilon h'(x)\cos(x)}{R_2(x)}+\frac{d\varepsilon R'_2(x)h'(x)\cos(x)}{R_2(x)^2}\right)\nonumber\\
   	&-\frac{h(x)}{2 R_2(x)^2}\int\!\!\!\!\!\!\!\!\!\; {}-{} \log\left(\frac{1}{ \left(R_2(x)-R_1(y)\right)^2+4R_2(x)R_1(y)\sin^2\left(\frac{x-y}{2}\right)}\right)\nonumber\\
   	&\quad\times \left[(R_2(x)R_1(y)+R'_2(x)R'_1(y))\sin(x-y)+(R_2(x)R'_1(y)-R'_2(x)R_1(y))\cos(x-y)\right] dy\nonumber\\
   	&+\frac{1}{2 R_2(x)}\int\!\!\!\!\!\!\!\!\!\; {}-{} \log\left(\frac{1}{ \left(R_2(x)-R_1(y)\right)^2+4R_2(x)R_1(y)\sin^2\left(\frac{x-y}{2}\right)}\right)\nonumber\\
   	&\quad\times \left[(h(x)R_1(y)+h'(x)R'_1(y))\sin(x-y)+(h(x)R'_1(y)-h'(x)R_1(y))\cos(x-y)\right] dy\nonumber\\
   	&-\frac{1}{2 R_2(x)}\int\!\!\!\!\!\!\!\!\!\; {}-{} \frac{2\left(R_2(x)-R_1(y)\right)h(x)+4h(x)R_1(y)\sin^2\left(\frac{x-y}{2}\right)}{ \left(R_2(x)-R_1(y)\right)^2+4R_2(x)R_1(y)\sin^2\left(\frac{x-y}{2}\right)}\nonumber\\
   	&\quad\times \left[(R_2(x)R_1(y)+R'_2(x)R'_1(y))\sin(x-y)+(R_2(x)R'_1(y)-R'_2(x)R_1(y))\cos(x-y)\right] dy\nonumber\\
   	&+\frac{(1-\gamma)h(x)}{ 2 R_2(x)^2}\int\!\!\!\!\!\!\!\!\!\; {}-{}\log\left(\frac{1}{ \left(R_2(x)-R_2(y)\right)^2+4R_2(x)R_2(y)\sin^2\left(\frac{x-y}{2}\right)}\right)\nonumber\\
   	&\quad\times \left[(R_2(x)R_2(y)+R'_2(x)R'_2(y))\sin(x-y)+(R_2(x)R'_2(y)-R'_2(x)R_2(y))\cos(x-y)\right]dy\nonumber\\
   	&+\frac{\gamma-1}{ 2 R_2(x)}\int\!\!\!\!\!\!\!\!\!\; {}-{}\log\left(\frac{1}{ \left(R_2(x)-R_2(y)\right)^2+4R_2(x)R_2(y)\sin^2\left(\frac{x-y}{2}\right)}\right)\nonumber\\
   	&\quad\times \left[(h(x)R_2(y)+R_2(x)h(y)+h'(x)R'_2(y)+R'_2(x)h'(y))\sin(x-y)\right.\\
   	&\quad\quad\left.+(h(x)R'_2(y)+R_2(x)h'(y)-h'(x)R_2(y)-R'_2(x)h(y))\cos(x-y)\right]dy\nonumber\\
   	&+\frac{1-\gamma}{ 2 R_2(x)}\int\!\!\!\!\!\!\!\!\!\; {}-{} \frac{2\left(R_2(x)-R_2(y)\right)(h(x)-h(y))+4\left(h(x)R_2(y)+R_2(x)h(y)\right)\sin^2\left(\frac{x-y}{2}\right)}{ \left(R_2(x)-R_2(y)\right)^2+4R_2(x)R_2(y)\sin^2\left(\frac{x-y}{2}\right)} \nonumber\\
   	&\quad\times \left[(R_2(x)R_2(y)+R'_2(x)R'_2(y))\sin(x-y)+(R_2(x)R'_2(y)-R'_2(x)R_2(y))\cos(x-y)\right]dy\nonumber\\
   	&+\sum_{n=1}^{N-1} \frac{(\gamma-1)h(x)}{ R_2(x)^2} \int\!\!\!\!\!\!\!\!\!\; {}-{}\log\left( \frac{1}{\left| \boldsymbol z_2(x)-d\boldsymbol e_1- Q_{\frac{2\pi n}{N}} \left(\boldsymbol z_1(y)-d\boldsymbol e_1\right)\right|}\right)\nonumber\\
   	&\quad \times \left[(R_2(x)R_1(y)+R_2'(x)R_1'(y))\sin(x-y-\frac{2\pi n}{N})\right.\\
   	&\qquad\left.+(R_2(x)R_1'(y)-R_2'(x)R_1(y))\cos(x-y-\frac{2\pi n}{N})\right] dy\nonumber\\
   	&+\sum_{n=1}^{N-1} \frac{1-\gamma}{ R_2(x)} \int\!\!\!\!\!\!\!\!\!\; {}-{}\log\left( \frac{1}{\left| \boldsymbol z_2(x)-d\boldsymbol e_1- Q_{\frac{2\pi n}{N}} \left(\boldsymbol z_1(y)-d\boldsymbol e_1\right)\right|}\right)\nonumber\\
   	&\qquad \times \left[(h(x)R_1(y)+h'(x)R_1'(y))\sin(x-y-\frac{2\pi n}{N})\right.\\
   	&\qquad\left.+(h(x)R_1'(y)-h'(x)R_1(y))\cos(x-y-\frac{2\pi n}{N})\right] dy\nonumber\\
   	&+\sum_{n=1}^{N-1} \frac{\gamma-1}{ 2R_2(x)} \int\!\!\!\!\!\!\!\!\!\; {}-{}\frac{\partial_{f_2}B_{21n}h }{\left|\boldsymbol z_2(x)-d\boldsymbol e_1- Q_{\frac{2\pi n}{N}} \left(\boldsymbol z_1(y)-d\boldsymbol e_1\right)\right|^2}\nonumber\\
   	&\quad \times \left[(R_2(x)R_1(y)+R_2'(x)R_1'(y))\sin(x-y-\frac{2\pi n}{N})\right.\\
   	&\qquad\left.+(R_2(x)R_1'(y)-R_2'(x)R_1(y))\cos(x-y-\frac{2\pi n}{N})\right] dy\nonumber\\
   	&+\sum_{n=1}^{N-1} \frac{(1-\gamma)h(x)}{R_2(x)^2} \int\!\!\!\!\!\!\!\!\!\; {}-{}\log\left( \frac{1}{\left| \boldsymbol z_2(x)-d\boldsymbol e_1- Q_{\frac{2\pi n}{N}} \left(\boldsymbol z_2(y)-d\boldsymbol e_1\right)\right|}\right)\nonumber\\
   	&\quad\,\times \left[(R_2(x)R_2(y)+R_2'(x)R_2'(y))\sin(x-y-\frac{2\pi n}{N})\right.\\
   	&\qquad\left.+(R_2(x)R_2'(y)-R_2'(x)R_2(y))\cos(x-y-\frac{2\pi n}{N})\right] dy\nonumber\\
   	&+\sum_{n=1}^{N-1} \frac{\gamma-1}{ R_2(x)} \int\!\!\!\!\!\!\!\!\!\; {}-{}\log\left( \frac{1}{\left| \boldsymbol z_2(x)-d\boldsymbol e_1- Q_{\frac{2\pi n}{N}} \left(\boldsymbol z_2(y)-d\boldsymbol e_1\right)\right|}\right)\nonumber\\
   	&\quad\,\times \left[(h(x)R_2(y)+R_2(x)h(y)+h'(x)R_2'(y)+R_2'(x)h'(y))\sin(x-y-\frac{2\pi n}{N})\right.\nonumber\\
   	&\quad\quad\left.+(h(x)R_2'(y)+R_2(x)h'(y)-h'(x)R_2(y)-R_2'(x)h(y))\cos(x-y-\frac{2\pi n}{N})\right] dy\nonumber\\
   	&+\sum_{n=1}^{N-1} \frac{1-\gamma}{ R_2(x) } \int\!\!\!\!\!\!\!\!\!\; {}-{}\frac{\partial_{f_2}B_{22n}h}{\left|\boldsymbol z_2(x)-d\boldsymbol e_1- Q_{\frac{2\pi n}{N}} \left(\boldsymbol z_2(y)-d\boldsymbol e_1\right)\right|^2}\nonumber\\
   	&\quad\,\times \left[(R_2(x)R_2(y)+R_2'(x)R_2'(y))\sin(x-y-\frac{2\pi n}{N})\right.\\
   	&\qquad\left.+(R_2(x)R_2'(y)-R_2'(x)R_2(y))\cos(x-y-\frac{2\pi n}{N})\right] dy.\nonumber
   \end{align*}
\end{proof}

 Taking $\varepsilon=0$ and $\boldsymbol f=0$ in the derivatives $\partial_{f_i} G_j^0$,  using integration by parts and changing $x-y$ to $y$, we obtain
\begin{align}\label{3-12}
	&\quad\partial_{f_1}G_1^0 (0, \Omega,0) h\\
	&=\frac{1}{2}\int\!\!\!\!\!\!\!\!\!\; {}-{}\log\left(\frac{1}{4\sin^2\left(\frac{y}{2}\right)}\right) \left[h(x-y)\sin(y)+(h'(x-y)-h'(x))\cos(y)\right] dy\nonumber\\
	&-\frac{1}{2}\int\!\!\!\!\!\!\!\!\!\; {}-{} \sin(y)h(x-y)dy\nonumber\\
	& +\frac{\gamma-1}{2}\int\!\!\!\!\!\!\!\!\!\; {}-{}\log\left(\frac{1}{ \left(1-b\right)^2+4b\sin^2\left(\frac{y}{2}\right)}\right)
	\left[-h'(x)b\cos(y)\right]dy\nonumber\\
	&+\sum_{n=1}^{m-1}\int\!\!\!\!\!\!\!\!\!\; {}-{}\log\left( \frac{1}{A_n}\right) \left[h(x-y)\sin(y-\frac{2\pi n}{m})+h'(x-y)\cos(y-\frac{2\pi n}{m})\right] dy\nonumber\\
	&=\frac{1}{2}\int\!\!\!\!\!\!\!\!\!\; {}-{}\log\left(\frac{1}{\sin^2\left(\frac{y}{2}\right)}\right) \left[h(x-y)\sin(y)+(h'(x-y)-h'(x))\cos(y)\right] dy\nonumber\\
	&-\frac{1}{2}\int\!\!\!\!\!\!\!\!\!\; {}-{} \sin(y)h(x-y)dy\nonumber\\
	& +\frac{(1-\gamma)bh'(x)}{2}\int\!\!\!\!\!\!\!\!\!\; {}-{}\log\left(\frac{1}{ \left(1-b\right)^2+4b\sin^2\left(\frac{y}{2}\right)}\right)
	\cos(y)dy,\nonumber
\end{align}

\begin{align}\label{3-13}
	&\quad \partial_{f_2} G_1^0 (0, \Omega, 0) h\\
	&=\frac{\gamma-1}{2}\int\!\!\!\!\!\!\!\!\!\; {}-{}\log\left(\frac{1}{ \left(1-b\right)^2+4b\sin^2\left(\frac{x-y}{2}\right)}\right) \left[h(y)\sin(x-y)+h'(y)\cos(x-y)\right]dy\nonumber\\
	&+\frac{1-\gamma}{ 2}\int\!\!\!\!\!\!\!\!\!\; {}-{}\frac{b\sin(x-y)\left(-2\left(1-b\right)h(y)+4h(y)\sin^2\left(\frac{x-y}{2}\right)\right)}{ \left(1-b\right)^2+4b\sin^2\left(\frac{x-y}{2}\right)}\nonumber\\
	&+\sum_{n=1}^{N-1} \frac{\gamma-1}{ 2} \int\!\!\!\!\!\!\!\!\!\; {}-{}\log\left( \frac{1}{A_n}\right) \left[h(y)\sin(x-y-\frac{2\pi k}{m})+h'(y)\cos(x-y-\frac{2\pi n}{N})\right] dy\nonumber\\
	&=\frac{\gamma-1}{2}\int\!\!\!\!\!\!\!\!\!\; {}-{}\log\left(\frac{1}{ \left(1-b\right)^2+4b\sin^2\left(\frac{x-y}{2}\right)}\right) h(y)\sin(x-y)dy\nonumber\\
	&+\frac{1-\gamma}{2}\int\!\!\!\!\!\!\!\!\!\; {}-{}\partial_y\left(\log\left(\frac{1}{ \left(1-b\right)^2+4b\sin^2\left(\frac{x-y}{2}\right)}\right) \cos(x-y)\right)h(y)dy\nonumber\\
	&+\frac{1-\gamma}{ 2}\int\!\!\!\!\!\!\!\!\!\; {}-{}\frac{2b\sin(x-y)\left(b-\cos(y)\right)}{ \left(1-b\right)^2+4b\sin^2\left(\frac{x-y}{2}\right)}h(y)dy\nonumber\\
	&=(1-\gamma)b^2\int\!\!\!\!\!\!\!\!\!\; {}-{}\frac{\sin(x-y)h(y)}{ \left(1-b\right)^2+4b\sin^2\left(\frac{x-y}{2}\right)}dy,\nonumber
\end{align}

\begin{align}\label{3-14}
	&\quad \partial_{f_1} G_2^0 (0, \Omega, 0) h\\
	&=\frac{1}{2 }\int\!\!\!\!\!\!\!\!\!\; {}-{} \log\left(\frac{1}{ \left(1-b\right)^2+4b\sin^2\left(\frac{x-y}{2}\right)}\right)\left[h(y) \sin(x-y)+h'(y) \cos(x-y)\right] dy\nonumber\\
	&-\frac{1}{2}\int\!\!\!\!\!\!\!\!\!\; {}-{} \frac{-2\left(b-1\right)h(y)+4bh(y)\sin^2\left(\frac{x-y}{2}\right)}{ \left(1-b\right)^2+4b\sin^2\left(\frac{x-y}{2}\right)}\sin(x-y)dy\nonumber\\
	&+\sum_{n=1}^{m-1} \frac{1}{2} \int\!\!\!\!\!\!\!\!\!\; {}-{}\log\left( \frac{1}{A_n}\right) \left[h(y)\sin(x-y-\frac{2\pi k}{m})+h'(y)\cos(x-y-\frac{2\pi k}{m})\right] dy\nonumber\\
	&=-\int\!\!\!\!\!\!\!\!\!\; {}-{}\frac{\sin(x-y)h(y)}{ \left(1-b\right)^2+4b\sin^2\left(\frac{x-y}{2}\right)}dy.\nonumber
\end{align}

\begin{align}\label{3-15}
	&\quad\partial_{f_2}G_2^0 (0, \Omega, 0) h\\
	&=-\frac{1}{2 b}\int\!\!\!\!\!\!\!\!\!\; {}-{} \log\left(\frac{1}{ \left(1-b\right)^2+4b\sin^2\left(\frac{x-y}{2}\right)}\right)h'(x)\cos(x-y) dy\nonumber\\
	&\frac{\gamma-1}{ 2 }\int\!\!\!\!\!\!\!\!\!\; {}-{}\log\left(\frac{1}{ 4b^2\sin^2\left(\frac{x-y}{2}\right)}\right) \left[h(y)\sin(x-y)+(h'(y)-h'(x))\cos(x-y)\right]dy\nonumber\\
	&+\frac{1-\gamma}{ 2 }\int\!\!\!\!\!\!\!\!\!\; {}-{} h(y)\sin(x-y)dy.\nonumber
\end{align}
 Take $h_1,h_2\in H^{k}$ with
\begin{equation*}
	h_1(x)=\sum_{j=1}^\infty a_j \cos(jx),\quad h_2(x)=\sum_{j=1}^\infty b_j \cos(jx).
\end{equation*}

To compute the precise values of $\partial_{f_i}G_j^0 (0, \Omega, 0)$, we will need the following integral identities taken from \cite{Cas4}.
\begin{lemma}\label{lm3-3}
	For $0<b<1$, there hold
	\begin{equation}\label{3-16}
		\int\!\!\!\!\!\!\!\!\!\; {}-{} e^{i m y} \log\left(\frac{1}{\sin^2\left(\frac{y}{2}\right)}\right)=\begin{cases}\frac{1}{|m|},&\quad m\in \mathbb{Z}, m\not=0,\\2\log2, &\quad m=0,\end{cases}
	\end{equation}
    \begin{equation}\label{3-17}
    	\int\!\!\!\!\!\!\!\!\!\; {}-{} e^{i m y} \log\left(\frac{1}{(1-b)^2+4b\sin^2\left(\frac{y}{2}\right)}\right)=\begin{cases}\frac{1}{|m|}b^{|m|},&\quad m\in \mathbb{Z}, m\not=0,\\2\log(2b), &\quad m=0,\end{cases}
    \end{equation}
    \begin{equation}\label{3-18}
    	\int\!\!\!\!\!\!\!\!\!\; {}-{} \frac{e^{i m y}} {(1-b)^2+4b\sin^2\left(\frac{y}{2}\right)}=\frac{b^{|m|}}{1-b^2}.
    \end{equation}
\end{lemma}
\begin{proof}
	\eqref{3-16} is exactly Lemma 3.3 \cite{Cas4}.  \eqref{3-17} and \eqref{3-18} follows from Lemma 3.4 and Lemma 3.1 \cite{Cas4} respectively by taking $r=\frac{1-b}{1+b}$.
\end{proof}

We denote $G^0=(G^0_1, G^0_2)$, then the linearized operator of $G^0$ is as follows.
\begin{lemma}\label{lm3-4}
	\begin{align}\label{3-19}
		D G^0(0,\Om, 0) (h_1, h_2)=\begin{pmatrix}
			\partial_{f_1}G_1^0 (0, \Omega, 0) h_1 + \partial_{f_2}G_1^0 (0, \Omega, 0) h_2 \\
			\partial_{f_1}G_2^0(0, \Omega, 0) h_1 + \partial_{f_2}G_2^0 (0, \Omega, 0) h_2
		\end{pmatrix}
		=\sum_{j=1}^\infty  \frac{1}{2}M_j \begin{pmatrix} a_j \\ b_j \end{pmatrix} \sin(jx),
	\end{align}
	where $$M_j=\begin{pmatrix} j-1+(\gamma-1)j b^2 & (1-\gamma)b^{j+1} \\ -b^{j-1} & (\gamma-1)(j-1)+j \end{pmatrix}.$$
\end{lemma}
\begin{proof}
	We first compute $\partial_{f_1}G_1^\alpha (0, \Omega, 0) h_1$.
	Using \eqref{3-16}, we have
	\begin{align*}
		&\frac{1}{2}\int\!\!\!\!\!\!\!\!\!\; {}-{}\log\left(\frac{1}{\sin^2\left(\frac{y}{2}\right)}\right) \left[h_1(x-y)\sin(y)+(h_1'(x-y)-h_1'(x))\cos(y)\right] dy\nonumber\\
		&=\sum_{j=1}^\infty \frac{a_j}{2}\int\!\!\!\!\!\!\!\!\!\; {}-{}\log\left(\frac{1}{\sin^2\left(\frac{y}{2}\right)}\right) \left[\cos(j(x-y))\sin(y)-j\sin(j(x-y))\cos(y)+j\sin(jx)\cos(y)\right] dy\nonumber\\
		&=\sum_{j=1}^\infty \frac{a_j}{2}\int\!\!\!\!\!\!\!\!\!\; {}-{}\log\left(\frac{1}{\sin^2\left(\frac{y}{2}\right)}\right) \left[\cos(jx)\cos(jy)\sin(y)+\sin(jx)\sin(jy)\sin(y)\right.\\
		&\quad \left.-j\sin(jx)\cos(jy)\cos(y)+j\cos(jx)\sin(jy)\cos(y)+j\sin(jx)\cos(y)\right] dy\nonumber\\
		&=\sum_{j=1}^\infty \frac{a_j}{2}\int\!\!\!\!\!\!\!\!\!\; {}-{}\log\left(\frac{1}{\sin^2\left(\frac{y}{2}\right)}\right) \left[\sin(jx)\sin(jy)\sin(y)-j\sin(jx)\cos(jy)\cos(y)+j\sin(jx)\cos(y)\right] dy\nonumber\\
		&= \frac{a_1\sin(x)}{2}\int\!\!\!\!\!\!\!\!\!\; {}-{}\log\left(\frac{1}{\sin^2\left(\frac{y}{2}\right)}\right) \left[\sin(y)\sin(y)-\cos(y)\cos(y)+\cos(y)\right] dy\nonumber\\
		&\quad+\sum_{j=2}^\infty \frac{a_j\sin(jx)}{2}\int\!\!\!\!\!\!\!\!\!\; {}-{}\log\left(\frac{1}{\sin^2\left(\frac{y}{2}\right)}\right) \left[\sin(jy)\sin(y)-j\cos(jy)\cos(y)+j\cos(y)\right] dy\nonumber\\
		&= \frac{a_1\sin(x)}{2}\int\!\!\!\!\!\!\!\!\!\; {}-{}\log\left(\frac{1}{\sin^2\left(\frac{y}{2}\right)}\right) \left[-\cos(2y)+\cos(y)\right] dy\nonumber\\
		&\quad+\sum_{j=2}^\infty \frac{a_j\sin(jx)}{4}\int\!\!\!\!\!\!\!\!\!\; {}-{}\log\left(\frac{1}{\sin^2\left(\frac{y}{2}\right)}\right) \left[(1-j)\cos((j-1)y)-(j+1)\cos((j+1)y)+2j\cos(y)\right] dy\nonumber\\
		&= \frac{a_1\sin(x)}{2}\left(-\frac{1}{2}+1\right)+\sum_{j=2}^\infty \frac{a_j\sin(jx)}{4}(-1-1+2j)\nonumber\\
		&= \frac{a_1\sin(x)}{4}+\sum_{j=2}^\infty \frac{(j-1)a_j\sin(jx)}{2}.\nonumber\\
	\end{align*}
    By trigonometric formula, we drive
	\begin{align*}
		&-\frac{1}{2}\int\!\!\!\!\!\!\!\!\!\; {}-{} \sin(y)h_1(x-y)dy\nonumber\\
		&=\sum_{j=1}^\infty \frac{-a_j }{2}\int\!\!\!\!\!\!\!\!\!\; {}-{} \sin(y)\cos(j(x-y))dy\nonumber\\
		&=\sum_{j=1}^\infty \frac{-a_j }{4}\int\!\!\!\!\!\!\!\!\!\; {}-{} (\sin(jx+(1-j)y)-\sin((j+1)y-jx))dy\nonumber\\
		&= \frac{-a_1\sin(x)}{4}.\nonumber
	\end{align*}
    The remaining term follows by using \eqref{3-17}
	\begin{align*}
		\frac{bh_1'(x)}{2}\int\!\!\!\!\!\!\!\!\!\; {}-{}\log\left(\frac{1}{ \left(1-b\right)^2+4b\sin^2\left(\frac{y}{2}\right)}\right)
		\cos(y)dy=\frac{b^2 h_1'(x)}{2}.
	\end{align*}
	Thus, if we combine the above identities, we arrive at
	\begin{equation*}
		\partial_{f_1}G_1^\alpha (0, \Omega, 0) h_1=\sum_{j=1}^\infty \frac{j-1+(\gamma-1)jb^2}{2}a_j \sin(jx).
	\end{equation*}

    Now, we compute $\partial_{f_2} G_1^0 (0, \Omega, 0) h_2$. By \eqref{3-18}, we have
    \begin{align*}
    	&\quad \int\!\!\!\!\!\!\!\!\!\; {}-{}\frac{\sin(y)h_2(x-y)}{ \left(1-b\right)^2+4b\sin^2\left(\frac{y}{2}\right)}dy\\
    	&=\sum_{j=1}^\infty b_j\int\!\!\!\!\!\!\!\!\!\; {}-{}\frac{\sin(y)\cos(jx-jy)}{ \left(1-b\right)^2+4b\sin^2\left(\frac{y}{2}\right)}dy\\
    	&=\sum_{j=1}^\infty b_j\int\!\!\!\!\!\!\!\!\!\; {}-{}\frac{\cos(jx)\sin(y)\cos(jy)+\sin(jx)\sin(y)\sin(jy)}{ \left(1-b\right)^2+4b\sin^2\left(\frac{y}{2}\right)}dy\\
    	&=\sum_{j=1}^\infty b_j\sin(jx)\int\!\!\!\!\!\!\!\!\!\; {}-{}\frac{\sin(y)\sin(jy)}{ \left(1-b\right)^2+4b\sin^2\left(\frac{y}{2}\right)}dy\\
    	&=\sum_{j=1}^\infty \frac{b_j\sin(jx)}{2}\int\!\!\!\!\!\!\!\!\!\; {}-{}\frac{\cos((j-1)y)-\cos((j+1)y)}{ \left(1-b\right)^2+4b\sin^2\left(\frac{y}{2}\right)}dy\\
    	&=\sum_{j=1}^\infty \frac{b_j\sin(jx)}{2}\frac{1}{1-b^2}\left(b^{j-1}-b^{j+1}\right)\\
    	&=\sum_{j=1}^\infty \frac{b^{j-1}}{2}b_j\sin(jx).
    \end{align*}
    Thus, we obtain
    \begin{equation*}
    	\partial_{f_2}G_1^\alpha (0, \Omega, 0) h_2=\sum_{j=1}^\infty \frac{(1-\gamma)b^{j+1}}{2}b_j\sin(jx).
    \end{equation*}

    Using the above results, we deduce
    \begin{equation*}
    	\partial_{f_2}G_1^\alpha (0, \Omega, 0) h_2=\sum_{j=1}^\infty \frac{-b^{j-1}}{2}b_j\sin(jx).
    \end{equation*}
and
    \begin{equation*}
    	\partial_{f_2}G_2^\alpha (0, \Omega, 0) h_2=\sum_{j=1}^\infty \frac{(\gamma-1)(j-1)+j}{2}b_j \sin(jx).
    \end{equation*}
\end{proof}

We choose $b$ such that $DG^0(0, \Om, 0)$ is a isomorphism from $X_b^k\to Y_b^{k-1}$.
\begin{lemma}\label{lm3-5}
	For $\gamma\not=0$, there exists $b\in(0,1)$ such that $DG^0(0, \Om, 0):X_b^k\to Y_b^{k-1}$ is an isomorphism.
\end{lemma}
\begin{proof}
	By Lemma \ref{lm3-4} and $M_1=\begin{pmatrix} -(1-\gamma)b^2 & (1-\gamma )b^{2} \\ -1 & 1 \end{pmatrix}$, it is easy to see $DG^0(0, \Om, 0)$ maps  $X_b^k$ to $ Y_b^{k-1}$ for any $k\geq 3$. Next, we show the invertibility of $DG^0$.
	
	\emph{Case i):} $\gamma\in(0,1)$. We first consider the case $\gamma\in (0,1)$, where we assume $b\in(0, \frac{\sqrt 2}{2})$.
	
	The determinant of the matrix $M_j$ is
	\begin{align*}
		\det{M_j}=\gamma(1+(\gamma-1)b^2)j^2+((1-\gamma)(1+(\gamma-1)b^2)-\gamma)j+\gamma-1+(1-\gamma)b^{2j},
	\end{align*}
	for each $j\geq 2$.
	We define the auxiliary function
	\begin{align*}
		D(x)=\gamma(1+(\gamma-1)b^2)x^2+((1-\gamma)(1+(\gamma-1)b^2)-\gamma)x+\gamma-1.
	\end{align*}
	Suppose $0<b<\frac{\sqrt 2}{2}$, then for $x\geq 2$, we have
	\begin{align*}
		D'(x)&=2\gamma(1+(\gamma-1)b^2)x+(1-\gamma)(1+(\gamma-1)b^2)-\gamma\\
		&\geq 2\gamma(1-b^2)x+(1-\gamma)(1-b^2)-\gamma\\
		&\geq 4\gamma(1-b^2)+(1-\gamma)(1-b^2)-\gamma\\
		&\geq \gamma>0,
	\end{align*}
	which implies that $D(x)$ is strict increasing and
	\begin{align*}
		D(x)\geq D(2)&=4\gamma(1+(\gamma-1)b^2)+2((1-\gamma)(1+(\gamma-1)b^2)-\gamma)+\gamma-1\\
		&>4\gamma(1-b^2)+2((1-\gamma)(1-b^2)-\gamma)+\gamma-1\\
		&>2\gamma +1-\gamma -2\gamma +\gamma-1=0,
	\end{align*}
	for $x\geq 2$.
	Thus, $\det (M_j)>D(j)\geq D(2)>0$ for $j\geq 2$ and hence $M_j$ is invertible for $\gamma\in(0,1)$ and $b\in(0,\frac{\sqrt 2}{2})$.
	
	\emph{Case ii):} $\gamma>1$. In this case, we assume $b\in(0,1)$. Noticing that for all $\gamma>1$ and $0<b<1$, it holds $$\begin{cases}\gamma(1+(\gamma-1)b^2)>0\\
		\frac{(\gamma-1)(1+(\gamma-1)b^2)+\gamma}{2\gamma(1+(\gamma-1)b^2)}\leq 1\\\gamma-1+(1-\gamma)b^{2j}=(\gamma-1)(1-b^{2j})>0.\end{cases}$$
	Thus, one easily infer from the properties of quadratic function that $\det(M_j)>0$ for $j\geq 2$, which implies the invertiblility of $M_j$.
	
	\emph{Case iii):} $\gamma<0$. In this case, it is hard to give a precise interval, where $M_j$ is invertible for all $j\geq 2$ and $b$ belonging to this interval. Instead, we will give a abstract proof of the existence of  $b\in(0,1)$ such that $M_j$ is invertible. We claim  that for all $\gamma\in \mathbb{R}$ with $\gamma\not=0$, there exists a set $S_\gamma$ composed of at most  countable points such that as long as  $b\in (0,1)\setminus S_\gamma$, $M_j$ is invertible for any $j\geq 2$. Indeed, for $\gamma\not=0$, we have $\gamma(1+(\gamma-1)b^2)\not =0$ for all $b\in(0,1)$ except for at most one point $b_\gamma$. Thus, $D(x)$ is a quadratic polynomial of $x$ for $b\not= b_\gamma$. Noticing that $b^{2x}\rightarrow 0$ as $x\rightarrow +\infty$, for $b\in(0,1)\setminus\{b_\gamma\}$, there exists $x_\gamma(b)$ sufficiently large such that $\det(M_j)\not=0$ for all $j\geq x_\gamma(b)$. Moreover, $x_\gamma(b)$ is contious with respect to $b$. Take arbitrary closed interval $I\subset ((0,1)\setminus \{b_\gamma\})$. By the continuity of $x_\gamma(b)$, one can choose a universal $x_{\gamma, I}$ satisfying $\det(M_j)\not=0$ for all $j\geq x_{\gamma,I}$ and $b\in I$. For each fixed $j$, $\det(M_j)$ is a $2j$th degree polynomial of $b$ and hence the number of $b$ such that $\det(M_j)=0$ is at most $2j$. Therefore, there are finite many  $b$ such that $\det(M_j)$ is not invertible for $2\leq j\leq x_{\gamma, I}$. This is, $\det(M_j)$ is invertible for all $j\geq 2$ and $b\in I$ except for finite many $b$. Since $(0,1)\setminus\{b_\gamma\}$ is a union of countable closed intervals, the claim above follows easily.
	
	Now, we fix $b$ such that $\det(M_j)$ is invertible for all $j\geq 2$. Then,  $M_j^{-1}$ is given by
	$$M_j^{-1}=\frac{1}{\det(M_j)}\begin{pmatrix} (\gamma-1)(j-1)+j &  b^{j-1}\\ - (1-\gamma)b^{j+1}& j-1+(\gamma-1)j b^2, \end{pmatrix}.$$
	For $g=(g_1,g_2)\in  Y_b^{k-1}$ with $g_1=(1-\gamma)b^2 d_1\sin(x)+ \sum_{j=1}^\infty c_j\sin(jx)$ and $g_2=\sum_{j=1}^\infty d_j\sin{jx}$, we define its inverse $f=(f_1, f_2)$ with
	$$f^t=2\begin{pmatrix} (1-\gamma)b^2\\ 1\end{pmatrix}\frac{d_1}{1-(1-\gamma)b^2} \cos(x)+\sum_{j=2}^\infty 2M_j^{-1}\begin{pmatrix} c_j\\ d_j \end{pmatrix}\cos(jx).$$
	Noticing that $\det(M_j)=O(j^2)$, thus the elements of $M_j^{-1}$ is of order $j^{-1}$ and hence $f\in X_b^k$.
\end{proof}
\begin{remark}
   Lemma \ref{lm3-5} fails  for $\gamma=0$. In this case $M_j=\begin{pmatrix} j-1-j b^2 & b^{j+1} \\ -b^{j-1} & 1 \end{pmatrix}.$ Though $M_j$ is still invertible for $j\geq 2$ and $0<b<\frac{\sqrt 2}{2}$, the lower right element of $M_j^{-1}$ is of order $O(1)$ and hence $f$ defined as above does not belong to $X_b^k$.
\end{remark}

Our next step is to choose $\Om$ such that  $G^0=(G_1^0,G_2^0)$ maps $X_b^{k}$ to $Y_b^{k-1}$ for all $\varepsilon$ small. According to Lemma \ref{lm3-1},  we conclude that
\begin{align*}
	&\quad G^0_1(0, \Om, 0)=G^0_2=(1-b^2+\gamma b^2)\left(\pi\Omega d-\sum_{n=1}^{N-1} \frac{1-\cos(\frac{2\pi n}{N})}{2d\left((-1+\cos(\frac{2\pi n}{N}))^2 +\sin^2(\frac{2\pi n}{N})\right)}\right)\sin(x).
\end{align*}
Therefore, if we take $\Om$ to be
\begin{equation}\label{3-20}
	\Om_*^0:=\sum_{n=1}^{N-1} \frac{1-\cos(\frac{2\pi n}{N})}{2\pi d^2\left((-1+\cos(\frac{2\pi n}{N}))^2 +\sin^2(\frac{2\pi n}{N})\right)},
\end{equation}
then $ G^0_1(0, \Om_*^0, 0)= G^0_2(0, \Om_*^0, 0)=0$ and hence $G^0(0, \Om_*^0, 0)\in Y_b^{k-1}$.

For general $\varepsilon$ and $\boldsymbol f$, we can take $\Om$ properly such that the range of $G^0=(G_1^0,G_2^0)$ belongs to $Y_b^{k-1}$.

\begin{lemma}\label{lm3-6}
	There exists
	\begin{equation*}
		\Omega^0(\varepsilon, \boldsymbol f):=\Omega_*^0+\varepsilon\mathcal{R}_\Omega(\varepsilon,\boldsymbol f)
	\end{equation*}
	with $\Omega_*^0$ given by \eqref{3-20} and continuous function $\mathcal{R}_\Omega(\varepsilon, \boldsymbol f)$ such that $\tilde G^0(\varepsilon, \boldsymbol f):=G^0(\varepsilon,\Omega^0(\varepsilon,\boldsymbol f), \boldsymbol f)$ maps $(-\varepsilon_0,  \varepsilon_0)\times V$ to $Y_b^{k-1}$.
	Moreover, $\partial_{\boldsymbol f}\mathcal{R}_\Omega(\varepsilon,  \boldsymbol f):X^k\times X^k\to \mathbb{R}^2$ is continuous.
\end{lemma}
\begin{proof}
	It suffices to find $\Om$ such that
	\begin{equation}\label{3-21}
		\int\!\!\!\!\!\!\!\!\!\; {}-{} G^0_1(\varepsilon, \Om, \boldsymbol f)\sin(x) dx=(1-\gamma)b^2 \int\!\!\!\!\!\!\!\!\!\; {}-{} G^0_2(\varepsilon, \Om, \boldsymbol f)\sin(x) dx.
	\end{equation}
	For the sake of convenience, set $\Om=\Om_*^0+\Om_\varepsilon$ and assume $f_1=\sum_{j=1}^\infty a_j \cos(jx)$ and $f_2=\sum_{j=1}^\infty b_j \cos(jx)$ with $a_1=(1-\gamma)b^2 b_1$.
	By \eqref{3-1}, \eqref{3-9}, \eqref{3-10} and the calculations in Lemma \ref{lm3-4}, we have
	\begin{align*}
		&\int\!\!\!\!\!\!\!\!\!\; {}-{} G^0_1(\varepsilon, \Om, \boldsymbol f)\sin(x) dx=\pi(1-b^2+\gamma b^2)\Om \left(\frac{d}{2}+\varepsilon \int\!\!\!\!\!\!\!\!\!\; {}-{} \mathcal{R}_{11}(\varepsilon, \boldsymbol f)\sin(x)dx\right)+(\gamma-1)b^2(a_1-b_1)\\
		&\quad-\frac{(1-b^2+\gamma b^2)\Om_*^0 d}{2} +\int\!\!\!\!\!\!\!\!\!\; {}-{} \varepsilon\mathcal{R}_1(\varepsilon, \boldsymbol f)\sin(x) dx\\
		&=\pi(1-b^2+\gamma b^2)\Om_\varepsilon\left(\frac{d}{2}+\varepsilon \int\!\!\!\!\!\!\!\!\!\; {}-{} \mathcal{R}_{11}(\varepsilon, \boldsymbol f)\sin(x)dx\right) +(\gamma-1)b^2(a_1-b_1) \\
		&\quad +\varepsilon\int\!\!\!\!\!\!\!\!\!\; {}-{} \left(\pi(1-b^2+\gamma b^2)\Om_*^0 \mathcal{R}_{11}(\varepsilon, \boldsymbol f)+\mathcal{R}_1(\varepsilon, \boldsymbol f)\right)\sin(x) dx,
	\end{align*}
and
\begin{align*}
	&\int\!\!\!\!\!\!\!\!\!\; {}-{} G^0_2(\varepsilon, \Om, \boldsymbol f)\sin(x) dx=\pi(1-b^2+\gamma b^2)\Om\left(\frac{d}{2}+\varepsilon \int\!\!\!\!\!\!\!\!\!\; {}-{} \mathcal{R}_{21}(\varepsilon, \boldsymbol f)\sin(x)dx\right)+(-a_1+b_1)\\
	&\quad-\frac{(1-b^2+\gamma b^2)\Om_*^0 d}{2} +\int\!\!\!\!\!\!\!\!\!\; {}-{} \varepsilon\mathcal{R}_2(\varepsilon, \boldsymbol f)\sin(x) dx\\
	&=\pi(1-b^2+\gamma b^2)\Om_\varepsilon \left(\frac{d}{2}+\varepsilon \int\!\!\!\!\!\!\!\!\!\; {}-{} \mathcal{R}_{21}(\varepsilon, \boldsymbol f)\sin(x)dx\right) +(-a_1+b_1)\\ &\quad +\varepsilon\int\!\!\!\!\!\!\!\!\!\; {}-{} \left(\pi(1-b^2+\gamma b^2)\Om_*^0 \mathcal{R}_{21}(\varepsilon, \boldsymbol f)+\mathcal{R}_2(\varepsilon, \boldsymbol f)\right)\sin(x) dx.
\end{align*}
Then, \eqref{3-21} is equivalent to
\begin{align*}
	&\pi(1-b^2+\gamma b^2)\Om_\varepsilon \left(\frac{d}{2}+\varepsilon \int\!\!\!\!\!\!\!\!\!\; {}-{} \mathcal{R}_{11}(\varepsilon, \boldsymbol f)\sin(x)dx\right)+(\gamma-1)b^2(a_1-b_1)\\ &\quad +\varepsilon\int\!\!\!\!\!\!\!\!\!\; {}-{} \left(\pi(1-b^2+\gamma b^2)\Om_*^0 \mathcal{R}_{11}(\varepsilon, \boldsymbol f)+\mathcal{R}_1(\varepsilon, \boldsymbol f)\right)\sin(x) dx\\
	=&(1-\gamma)b^2\left(\pi(1-b^2+\gamma b^2)\Om_\varepsilon \left(\frac{d}{2}+\varepsilon \int\!\!\!\!\!\!\!\!\!\; {}-{} \mathcal{R}_{21}(\varepsilon, \boldsymbol f)\sin(x)dx\right) +(-a_1+b_1)\right.\\
	 &\quad \left.+\varepsilon\int\!\!\!\!\!\!\!\!\!\; {}-{} \left(\pi(1-b^2+\gamma b^2)\Om_*^0 \mathcal{R}_{21}(\varepsilon, \boldsymbol f)+\mathcal{R}_2(\varepsilon, \boldsymbol f)\right)\sin(x) dx\right),
\end{align*}
which implies
\[
\tiny{\Om_\varepsilon=\varepsilon\frac{-\int\!\!\!\!\!\!\!\!\!\; {}-{} \left(\pi(1-b^2+\gamma b^2)\Om_*^0 \mathcal{R}_{21}(\varepsilon, \boldsymbol f)+\mathcal{R}_2(\varepsilon, \boldsymbol f)\right)\sin(x) dx+(1-\gamma)b^2\int\!\!\!\!\!\!\!\!\!\; {}-{} \left(\pi(1-b^2+\gamma b^2)\Om_*^0 \mathcal{R}_{21}(\varepsilon, \boldsymbol f)+\mathcal{R}_2(\varepsilon, \boldsymbol f)\right)\sin(x) dx}{\pi(1-b^2+\gamma b^2)\left((1-(1-\gamma)b^2)d/2+\varepsilon\left( \int\!\!\!\!\!\!\!\!\!\; {}-{} \mathcal{R}_{11}(\varepsilon, \boldsymbol f)\sin(x)dx-(1-\gamma )b^2 \int\!\!\!\!\!\!\!\!\!\; {}-{} \mathcal{R}_{21}(\varepsilon, \boldsymbol f)\sin(x)dx\right)\right)}}.
\]

Hence, the desired $ \mathcal{R}_\Omega(\varepsilon,\boldsymbol f)$ is given by $\mathcal{R}_\Omega(\varepsilon,\boldsymbol f)= \frac{\Om_\varepsilon}{\varepsilon}$ and the regularity of $\mathcal{R}_\Omega(\varepsilon,\boldsymbol f)$ follows from Lemmas \ref{lm3-1} and \ref{lm3-2}.

The proof is thus complete.	
\end{proof}
Now, we are in the position to prove Theorem \ref{thm1}.

\noindent{\bf Proof of Theorem \ref{thm1}:}
We first prove $\partial_f\tilde G^0(0,0): X_b^k\to Y_b^{k-1}$ is an isomorphism. By chain rule, it holds
\begin{equation*}
	\partial_f\tilde G^0(0,0)h=\partial_\Omega G^0(0,\Omega^0_*,0)\partial_f\Omega^0(0,0)h+\partial_f G^0(0,\Omega^0_*,0).
\end{equation*}
From Lemma \ref{lm3-6}, we know $\partial_f\Omega^0(0,0)=0$. Hence $\partial_f\tilde G^0(0,0)=\partial_f G^0(0,\Omega^0_*,0)$, and we achieve the desired result by Lemma \ref{lm3-5}.

According to Lemma \ref{lm3-1}--Lemma \ref{lm3-6}, we can apply implicit function theorem to prove that there exists $\varepsilon_0>0$ such that
\begin{equation*}
	\left\{(\varepsilon,\boldsymbol f)\in [-\varepsilon_0,\varepsilon_0]\times V \ : \ \tilde G^0_i(\varepsilon,\boldsymbol f)=0,\quad \text{for}\, i=1,2\right\}
\end{equation*}
is parameterized by one-dimensional curve $\varepsilon\in [-\varepsilon_0,\varepsilon_0]\to (\varepsilon, \boldsymbol f_\varepsilon)$.

Let $\tilde {\boldsymbol f}(x)=\boldsymbol f(-x)$. Our last step is to show that if $(\varepsilon,\boldsymbol f)$ is a solution to $\tilde G^0_i(\varepsilon,f)=0$ for $i=1,2$, then so is $(-\varepsilon, \tilde {\boldsymbol f})$. By changing $y$ to $-y$ in the integral representation of $G^0(\varepsilon,\Omega,f)$ and using the fact that $f$ is an even function, we obtain $\Omega^0(\varepsilon,\boldsymbol f)=\Omega^0(-\varepsilon,\tilde{ \boldsymbol f})$. Then we can insert it into $G^0(\varepsilon,\Omega, \boldsymbol f)$ and derive $\tilde G^0(-\varepsilon, \tilde {\boldsymbol f})=0$ by a similar substitution of variables as before.

Hence the proof of Theorem \ref{thm1} is finished.

\section{Existence of co-rotating solutions: the case $0<\alpha<2$}
The proof for the case $0<\alpha<2$ follows the same argument as the case $\alpha=0$, so we will be sketch and omit the details.

Set $$V:=\begin{cases}\{f\in X_b^k\mid ||g||_{X^k}<1\},&\quad 0<\alpha<1,\\ \{f\in X_b^{k+\log}\mid ||g||_{X^{k+\log}}<1\},&\quad \alpha=1,\\ \{f\in X_b^{k+\alpha-1}\mid ||g||_{X^{k+\alpha-1}}<1\},&\quad 1<\alpha<2.\end{cases}
	$$
We have the following two lemmas concerning the regularity of $G_i^\alpha$.
\begin{lemma}\label{lm4-1}
	$G_i^\alpha(\varepsilon, \Omega, \textbf f): \left(-\varepsilon_0, \varepsilon_0\right)\times \mathbb{R} \times V \rightarrow Y^{k-1}$ is continuous for $i=1,2$.
\end{lemma}
\begin{proof}
	For $0<\alpha<1$, the proof is almost the same as the case $\alpha=0$ since the kernel $\frac{1}{\left| \left(R_1(x)-R_1(y)\right)^2+4R_1(x)R_1(y)\sin^2\left(\frac{x-y}{2}\right)\right|^{\frac{\alpha}{2}}}\in L^1$. For the case $1\leq \alpha <2$, the most singular terms has been treated in \cite{CQZZ} and hence  no additional difficulty  arises in the proof of this case. Further more, by Taylor's formula \eqref{3-4}, we have
	\begin{align}\label{4-1}
		&\quad G^\alpha_1=\pi(1-b^2+\gamma b^2)\Omega d\sin(x)\\
		&+C_\alpha(1-\frac{\alpha}{2})\int\!\!\!\!\!\!\!\!\!\; {}-{} \frac{f_1(x-y)\sin(y)dy}{\left(4\sin^2(\frac{y}{2})\right)^{\frac{\alpha}{2}}}+C_\alpha\int\!\!\!\!\!\!\!\!\!\; {}-{} \frac{(f_1'(x-y)-f_1'(x))\cos(y)dy}{\left(4\sin^2(\frac{y}{2})\right)^{\frac{\alpha}{2}}}\nonumber\\
		&+\frac{(1-\gamma)\alpha C_\alpha}{2}\int\!\!\!\!\!\!\!\!\!\; {}-{}\frac{b\left(-2(1-b)f_2(y)+4f_2(y)\sin^2\left(\frac{x-y}{2}\right)\right)\sin(x-y)}{\left|(1-b)^2+4b\sin^2\left(\frac{x-y}{2}\right)\right|^{\frac{\alpha}{2}+1}}dy\nonumber\\
		&+(1-\gamma)C_\alpha\int\!\!\!\!\!\!\!\!\!\; {}-{}\frac{f_2(y)\sin(x-y)+(f_2'(y)-bf_1'(x))\cos(x-y))}{\left|(1-b)^2+4b\sin^2\left(\frac{x-y}{2}\right)\right|^{\frac{\alpha}{2}}}dy\nonumber\\
		&+(b^2-\gamma b^2-1)\sum_{n=1}^{N-1} \frac{\alpha C_\alpha(1-\cos(\frac{2\pi n}{N}))\sin (x)}{2\left((-1+\cos(\frac{2\pi n}{N}))^2 +\sin^2(\frac{2\pi n}{N})\right)^{1+\frac{\alpha}{2}}d^{1+\alpha}}+\varepsilon|\varepsilon|^\alpha\mathcal{R}_{1}^\alpha(\varepsilon,\Om, \boldsymbol f),\nonumber
	\end{align}
and
  \begin{align}\label{4-2}
	&\quad G^\alpha_2=\pi(1-b^2+\gamma b^2)\Omega d\sin(x)\\
	&=-(1-\gamma)b^{-\alpha}\left(C_\alpha(1-\frac{\alpha}{2})\int\!\!\!\!\!\!\!\!\!\; {}-{} \frac{f_2(x-y)\sin(y)dy}{\left(4\sin^2(\frac{y}{2})\right)^{\frac{\alpha}{2}}}+C_\alpha\int\!\!\!\!\!\!\!\!\!\; {}-{} \frac{(f_2'(x-y-)f_2'(x))\cos(y)dy}{\left(4\sin^2(\frac{y}{2})\right)^{\frac{\alpha}{2}}}\right)\nonumber\\
	&-\frac{\alpha C_\alpha}{2}\int\!\!\!\!\!\!\!\!\!\; {}-{}\frac{\left(2(1-b)f_1(y)+4bf_1(y)\sin^2\left(\frac{x-y}{2}\right)\right)\sin(x-y)}{\left|(1-b)^2+4b\sin^2\left(\frac{x-y}{2}\right)\right|^{\frac{\alpha}{2}+1}}dy\nonumber\\
	&+\frac{C_\alpha}{b}\int\!\!\!\!\!\!\!\!\!\; {}-{}\frac{bf_1(y)\sin(x-y)+(bf_1'(y)-f_2'(x))\cos(x-y))}{\left|(1-b)^2+4b\sin^2\left(\frac{x-y}{2}\right)\right|^{\frac{\alpha}{2}}}dy\nonumber\\
	&+(b^2-\gamma b^2-1)\sum_{n=1}^{N-1} \frac{\alpha C_\alpha(1-\cos(\frac{2\pi n}{N}))\sin (x)}{2\left((-1+\cos(\frac{2\pi n}{N}))^2 +\sin^2(\frac{2\pi n}{N})\right)^{1+\frac{\alpha}{2}}d^{1+\alpha}}+\varepsilon|\varepsilon|^\alpha \mathcal{R}_{2}^\alpha(\varepsilon,\Om, \boldsymbol f).\nonumber
\end{align}
\end{proof}

\begin{lemma}\label{lm4-2}
	The Gautex derivative $\partial_{f_i} G_j^\alpha(\varepsilon, \Omega, \boldsymbol f): \left(-\varepsilon_0, \varepsilon_0\right)\times \mathbb{R} \times V \rightarrow L(X^k, Y^{k-1})$ exists and  is continuous for $i, j=1,2$.	
\end{lemma}
\begin{proof}
	For the case $0<\alpha<1$, the proof is almost the same as the case $\alpha=0$. For the case $1\leq \alpha <2$, the most singular terms are treated in \cite{CQZZ} and hence the proof of this case contains no additional difficulty.
\end{proof}

 To calculate the exact expression of the Gautex derivatives, we will follow the  notations of \cite{de1}. Let $F(a,b,c, z)$ be the hypergeometric function, whose definition and properties can be found in \cite{de1, Rai}.

 Denote $$\Theta_j=\begin{cases}2^{\alpha-1}\frac{ \Gamma(1-\alpha)}{(\Gamma(1-\frac{\alpha}{2}))^2}\left(\frac{\Gamma(1+\frac{\alpha}{2})}{\Gamma(2-\frac{\alpha}{2})}-\frac{\Gamma(j+\frac{\alpha}{2})}{\Gamma(1+j-\frac{\alpha}{2})}\right),&\quad \alpha\not=1\\
	\frac{2}{\pi}\sum_{l=2}^j,&\quad \alpha=1,\end{cases}$$
$$(x)_j:=\begin{cases} 1, &j=0\\ x(x+1)(x+2)\cdots(x+j-1),\quad &j\geq 1.\end{cases}$$
and
$$\Lambda_j(b):=\frac{1}{b}\int_0^\infty \frac{J_n(bt) J_n(t)}{t^{1-\alpha}}dt,$$
where $J_n$ represents the Bessel function of the first kind.
The relationship between $\Lambda_j$ and the hypergeometric function is (see (5.24) \cite{de1})
$$\Lambda_j=\frac{\Gamma\left(\frac{\alpha}{2}\right)}{2^{1-\alpha}\Gamma\left(1-\frac{\alpha}{2}\right)} \frac{\left(\frac{\alpha}{2}\right)_j}{j!} b^{j-1} F\left(\left(\frac{\alpha}{2}\right), j+\frac{\alpha}{2}, j+1, b^2\right).$$

 Take $h_1,h_2$ with
\begin{equation*}
	h_1(x)=\sum_{j=1}^\infty a_j \cos(jx),\quad h_2(x)=\sum_{j=1}^\infty b_j \cos(jx).
\end{equation*}
The linearized operator of $G^\alpha$ is as follows.
\begin{lemma}\label{lm4-3}
	\begin{align}\label{4-3}
		D G^\alpha(0,\Om, 0) (h_1,h_2)=\begin{pmatrix}
			\partial_{f_1}G_1^\alpha (0, \Omega, 0) h_1 + \partial_{f_2}G_1^\alpha (0, \Omega, 0) h_2 \\
			\partial_{f_1}G_2^\alpha (0, \Omega, 0) h_1 + \partial_{f_2}G_2^\alpha (0, \Omega, 0) h_2
		\end{pmatrix}
		=\sum_{j=1}^\infty  j M_j \begin{pmatrix} a_j \\ b_j \end{pmatrix} \sin(jx),
	\end{align}
	where $$M_j=\begin{pmatrix} \Theta_j+(\gamma-1)b^2 \Lambda_1(b) & (1-\gamma)b^2 \Lambda_j(b) \\ -\Lambda_j(b) & (\gamma-1)b^{-\alpha}\Theta_j+\Lambda_1(b) \end{pmatrix}.$$
\end{lemma}
\begin{proof}
	By direct calculations, we obtain
	\begin{align}\label{4-5}
		&\quad\partial_{f_1}G_1^\alpha (0, \Omega, 0) h\\
		&=\left(1-\frac{\alpha}{2}\right)C_\alpha\int\!\!\!\!\!\!\!\!\!\; {}-{} \frac{h(x-y)\sin(y)}{\left| 4\sin^2\left(\frac{y}{2}\right)\right|^{\frac{\alpha}{2}}}dy+C_\alpha\int\!\!\!\!\!\!\!\!\!\; {}-{} \frac{(h'(x-y)-h'(x))\cos(y)}{\left| 4\sin^2\left(\frac{y}{2}\right)\right|^{\frac{\alpha}{2}}}dy\nonumber\\
		&\quad +(1-\gamma)C_\alpha\int\!\!\!\!\!\!\!\!\!\; {}-{}\frac{bh'(x)\cos(y)}{\left| \left(1-b\right)^2+4b\sin^2\left(\frac{y}{2}\right)\right|^{\frac{\alpha}{2}}}dy\nonumber
	\end{align}

	\begin{align}\label{4-6}
		\partial_{f_2}G_1^\alpha (0, \Omega, 0) h&=(\gamma-1)C_\alpha\int\!\!\!\!\!\!\!\!\!\; {}-{}\frac{h(x-y)\sin(y)+h'(x-y)\cos(y)}{\left| \left(1-b\right)^2+4b\sin^2\left(\frac{y}{2}\right)\right|^{\frac{\alpha}{2}}}dy\\
		&+(1-\gamma )\alpha C_\alpha\int\!\!\!\!\!\!\!\!\!\; {}-{}\frac{bh(x-y)\sin(y)\left(b-\cos(y)\right)}{\left| \left(1-b\right)^2+4b\sin^2\left(\frac{y}{2}\right)\right|^{\frac{\alpha}{2}+1}}dy\nonumber
	\end{align}
\begin{align}\label{4-7}
	\partial_{f_1}G_2^\alpha (0, \Omega, 0) h
	&=C_\alpha\int\!\!\!\!\!\!\!\!\!\; {}-{} \frac{h(x-y)\sin(y)+h'(x-y)\cos(y)}{\left| \left(1-b\right)^2+4b\sin^2\left(\frac{y}{2}\right)\right|^{\frac{\alpha}{2}}}dy\\
	&-\alpha C_\alpha\int\!\!\!\!\!\!\!\!\!\; {}-{} \frac{h(x-y)\sin(y)\left(1-b\cos(y)\right) dy}{\left| \left(1-b\right)^2+4b\sin^2\left(\frac{y}{2}\right)\right|^{\frac{\alpha}{2}+1}}dy\nonumber
\end{align}

\begin{align}\label{4-8}
	&\quad\partial_{f_2}G_2^\alpha (0, \Omega, 0) h\\
	&=(1-\gamma)\left(\frac{\alpha }{ 2}-1\right)C_\alpha\int\!\!\!\!\!\!\!\!\!\; {}-{}\frac{\sin(y)h(x-y)}{\left| 4b^2\sin^2\left(\frac{y}{2}\right)\right|^{\frac{\alpha}{2}}}dy+(\gamma-1)C_\alpha\int\!\!\!\!\!\!\!\!\!\; {}-{}\frac{(h'(x-y)-h'(x))\cos(y)}{\left| 4b^2\sin^2\left(\frac{y}{2}\right)\right|^{\frac{\alpha}{2}}}dy\nonumber\\
	&-\frac{C_\alpha h'(x)}{b}\int\!\!\!\!\!\!\!\!\!\; {}-{} \frac{\cos(y)}{\left| \left(1-b\right)^2+4b\sin^2\left(\frac{y}{2}\right)\right|^{\frac{\alpha}{2}}}dy\nonumber
\end{align}

We first calculate $\partial_{f_1}G_1^\alpha (0, \Omega, 0) h_1$.
The first two terms has been obtained in \cite{Has} for $0<\alpha<1$ and \cite{Cas1} for $1\leq \alpha<2$,
\begin{align*}
	&\quad \left(1-\frac{\alpha}{2}\right)C_\alpha\int\!\!\!\!\!\!\!\!\!\; {}-{} \frac{h_1(x-y)\sin(y)}{\left| 4\sin^2\left(\frac{y}{2}\right)\right|^{\frac{\alpha}{2}}}dy+C_\alpha\int\!\!\!\!\!\!\!\!\!\; {}-{} \frac{(h_1'(x-y)-h_1'(x))\cos(y)}{\left| 4\sin^2\left(\frac{y}{2}\right)\right|^{\frac{\alpha}{2}}}dy\\
	&=\sum_{j=1}^\infty j\Theta_j a_j \sin(jx),
\end{align*}
where $\Theta_j=\begin{cases}2^{\alpha-1}\frac{ \Gamma(1-\alpha)}{(\Gamma(1-\frac{\alpha}{2}))^2}\left(\frac{\Gamma(1+\frac{\alpha}{2})}{\Gamma(2-\frac{\alpha}{2})}-\frac{\Gamma(j+\frac{\alpha}{2})}{\Gamma(1+j-\frac{\alpha}{2})}\right),&\quad \alpha\not=1\\
\frac{2}{\pi}\sum_{l=2}^j,&\quad \alpha=1.\end{cases}$

Using (3.19) in \cite{de1},
\begin{align*}
	&\quad C_\alpha h'_1(x)b\int\!\!\!\!\!\!\!\!\!\; {}-{}\frac{\cos(y)}{\left| \left(1-b\right)^2+4b\sin^2\left(\frac{y}{2}\right)\right|^{\frac{\alpha}{2}}}dy\\
	&= C_\alpha h'_1(x)b^2\frac{\alpha}{2} F\left(\frac{\alpha}{2}, 1+\frac{\alpha}{2}, 2, b^2\right)\\
	&=b^2 \Lambda_1(b) h'_1(x).
\end{align*}

\begin{align}\label{4-9}
	&\quad \partial_{f_1}G_1^\alpha (0, \Omega, 0) h_1\\
	&=\sum_{j=1}^{\infty} j\left(\Theta_j+(\gamma-1)b^2 \Lambda_1(b)\right) a_j \sin(jx).\nonumber
\end{align}
Next, we consider the second term $\partial_{f_2}G_1^\alpha (0, \Omega, 0) h_2$. By (3.19) in \cite{de1}, we have
\begin{align*}
	&\quad -C_\alpha\int\!\!\!\!\!\!\!\!\!\; {}-{}\frac{h_2(x-y)\sin(y)}{\left| \left(1-b\right)^2+4b\sin^2\left(\frac{y}{2}\right)\right|^{\frac{\alpha}{2}}}dy\\
	&=-C_\alpha \sum_{j=1}^\infty b_j\int\!\!\!\!\!\!\!\!\!\; {}-{}\frac{\cos(jx-jy)\cos(y)}{\left| \left(1-b\right)^2+4b\sin^2\left(\frac{y}{2}\right)\right|^{\frac{\alpha}{2}}}dy\\
	&=-C_\alpha \sum_{j=1}^\infty b_j\int\!\!\!\!\!\!\!\!\!\; {}-{}\frac{\cos(jx)\cos(jy)\sin(y)+\sin(jx)\sin(jy)\sin(y)}{\left| \left(1-b\right)^2+4b\sin^2\left(\frac{y}{2}\right)\right|^{\frac{\alpha}{2}}}dy\\
	&=-C_\alpha b_j \sum_{j=1}^\infty \sin(jx)\int\!\!\!\!\!\!\!\!\!\; {}-{}\frac{\sin(jy)\sin(y)}{\left| \left(1-b\right)^2+4b\sin^2\left(\frac{y}{2}\right)\right|^{\frac{\alpha}{2}}}dy.
\end{align*}
Using integration by parts, we obtain
\begin{align*}
	&\quad C_\alpha\int\!\!\!\!\!\!\!\!\!\; {}-{}\frac{-h_2'(x-y)\cos(y)}{\left| \left(1-b\right)^2+4b\sin^2\left(\frac{y}{2}\right)\right|^{\frac{\alpha}{2}}}dy\\
	&=C_\alpha \sum_{j=1}^\infty  b_j\int\!\!\!\!\!\!\!\!\!\; {}-{}\frac{j\sin(jx-jy)\cos(y)}{\left| \left(1-b\right)^2+4b\sin^2\left(\frac{y}{2}\right)\right|^{\frac{\alpha}{2}}}dy\\
	&=C_\alpha \sum_{j=1}^\infty b_j\int\!\!\!\!\!\!\!\!\!\; {}-{}\frac{j\sin(jx)\cos(jy)\cos(y)-j\cos(jx)\sin(jy)\cos(y)}{\left| \left(1-b\right)^2+4b\sin^2\left(\frac{y}{2}\right)\right|^{\frac{\alpha}{2}}}dy\\
	&=C_\alpha \sum_{j=1}^\infty b_j j\sin(jx) \int\!\!\!\!\!\!\!\!\!\; {}-{}\frac{\cos(jy)\cos(y)}{\left| \left(1-b\right)^2+4b\sin^2\left(\frac{y}{2}\right)\right|^{\frac{\alpha}{2}}}dy\\
	&=C_\alpha \sum_{j=1}^\infty b_j\sin(jx) \int\!\!\!\!\!\!\!\!\!\; {}-{}\frac{\cos(y)}{\left| \left(1-b\right)^2+4b\sin^2\left(\frac{y}{2}\right)\right|^{\frac{\alpha}{2}}}d\sin(jy)\\
	&=C_\alpha \sum_{j=1}^\infty b_j \sin(jx) \int\!\!\!\!\!\!\!\!\!\; {}-{}\left(\frac{\sin(jy)\sin(y)}{\left| \left(1-b\right)^2+4b\sin^2\left(\frac{y}{2}\right)\right|^{\frac{\alpha}{2}}}+\frac{\alpha b}{2}\frac{\sin({jy})\sin(2y)}{\left| \left(1-b\right)^2+4b\sin^2\left(\frac{y}{2}\right)\right|^{\frac{\alpha}{2}+1}}\right)dy.
\end{align*}
By direct calculation and (3.19) in \cite{de1}, we derive
\begin{align*}
	&\quad -\alpha bC_\alpha\int\!\!\!\!\!\!\!\!\!\; {}-{}\frac{h_2(x-y)\sin(y)\cos(y)}{\left| \left(1-b\right)^2+4b\sin^2\left(\frac{y}{2}\right)\right|^{\frac{\alpha}{2}+1}}dy\\
	&=-\frac{\alpha bC_\alpha}{2}\sum_{j=1}^\infty b_j \int\!\!\!\!\!\!\!\!\!\; {}-{}\frac{\cos(j(x-y))\sin(2y)}{\left| \left(1-b\right)^2+4b\sin^2\left(\frac{y}{2}\right)\right|^{\frac{\alpha}{2}+1}}dy\\
	&=-\frac{\alpha bC_\alpha}{2}\sum_{j=1}^\infty b_j \int\!\!\!\!\!\!\!\!\!\; {}-{}\frac{\cos(jx)\cos(jy)\sin(2y)+\sin({jx})\sin({jy})\sin(2y)}{\left| \left(1-b\right)^2+4b\sin^2\left(\frac{y}{2}\right)\right|^{\frac{\alpha}{2}+1}}dy\\
	&=-\frac{\alpha bC_\alpha}{2}\sum_{j=1}^\infty b_j \sin({jx})\int\!\!\!\!\!\!\!\!\!\; {}-{}\frac{\sin({jy})\sin(2y)}{\left| \left(1-b\right)^2+4b\sin^2\left(\frac{y}{2}\right)\right|^{\frac{\alpha}{2}+1}}dy.
\end{align*}
Thus, we conclude by combining the above equalities that
\begin{align*}
	&\quad \partial_{f_2}G_1^\alpha (0, \Omega, 0) h_2\\
	&=(1-\gamma)\alpha b^2C_\alpha\int\!\!\!\!\!\!\!\!\!\; {}-{}\frac{h_2(x-y)\sin(y)}{\left| \left(1-b\right)^2+4b\sin^2\left(\frac{y}{2}\right)\right|^{\frac{\alpha}{2}+1}}\\
	&=(1-\gamma)\alpha b^2C_\alpha \sum_{j=1}^\infty b_j \sin(jx)\int\!\!\!\!\!\!\!\!\!\; {}-{}\frac{\sin(jy)\sin(y)}{\left| \left(1-b\right)^2+4b\sin^2\left(\frac{y}{2}\right)\right|^{\frac{\alpha}{2}+1}}dy\\
	&=(1-\gamma)\alpha b^2\frac{C_\alpha}{2} \sum_{j=1}^\infty b_j\sin(jx) \int\!\!\!\!\!\!\!\!\!\; {}-{}\frac{\cos((j-1)y)-\cos((j+1)y)}{\left| \left(1-b\right)^2+4b\sin^2\left(\frac{y}{2}\right)\right|^{\frac{\alpha}{2}+1}}dy\\
	&=(1-\gamma)\alpha b^2\frac{C_\alpha}{2} \sum_{j=1}^\infty b_j\sin(jx) \left[b^{j-1}\frac{\left(\frac{\alpha}{2}+1\right)_{(j-1)}}{(j-1)!}F(\frac{\alpha}{2}+1, j+\frac{\alpha}{2}, j, b^2)\right.\\
	&\quad \left.-b^{j+1}\frac{\left(\frac{\alpha}{2}+1\right)_{(j+1)}}{(j+1)!}F(\frac{\alpha}{2}+1, j+2+\frac{\alpha}{2}, j+2, b^2)\right]\\
	&=(1-\gamma)\alpha b^2\frac{C_\alpha}{2} \sum_{j=1}^\infty b_j \sin(jx)b^{j-1}\frac{\left(\frac{\alpha}{2}+1\right)_{(j-1)}}{(j-1)!} \left[F(\frac{\alpha}{2}+1, j+\frac{\alpha}{2}, j, b^2)\right.\\
	&\quad \left.-b^2\frac{\left(\frac{\alpha}{2}+j\right)\left(\frac{\alpha}{2}+j+1\right)}{j(j+1)}F(\frac{\alpha}{2}+1, j+2+\frac{\alpha}{2}, j+2, b^2)\right].\\
\end{align*}
Using (3.11), (3.14) and (3.14)-(3.15) in \cite{de1}, we have
\begin{align*}
	&\quad b^2\frac{\left(\frac{\alpha}{2}+j\right)\left(\frac{\alpha}{2}+j+1\right)}{j(j+1)}F(\frac{\alpha}{2}+1, j+2+\frac{\alpha}{2}, j+2, b^2)\\
	&=\frac{\frac{\alpha}{2}+j}{j}\left[F(\frac{\alpha}{2}+1, j+1+\frac{\alpha}{2}, j+1, b^2)-F(\frac{\alpha}{2}, j+1+\frac{\alpha}{2}, j+1, b^2)\right]\\
	&=\frac{\frac{\alpha}{2}+j}{j}\left[\frac{j}{\frac{\alpha}{2}+j}F(\frac{\alpha}{2}+1, j+\frac{\alpha}{2}, j, b^2)+\frac{\frac{\alpha}{2}}{j+\frac{\alpha}{2}}F(\frac{\alpha}{2}+1, j+\frac{\alpha}{2}, j+1, b^2)\right]\\
	&\quad -\frac{\frac{\alpha}{2}+j}{j}F(\frac{\alpha}{2}, j+1+\frac{\alpha}{2}, j+1, b^2)\\
	&=F(\frac{\alpha}{2}+1, j+\frac{\alpha}{2}, j, b^2)\\
	&\quad+\frac{\frac{\alpha}{2}}{j}F(\frac{\alpha}{2}+1, j+\frac{\alpha}{2}, j+1, b^2)-\frac{\frac{\alpha}{2}+j}{j}F(\frac{\alpha}{2}, j+1+\frac{\alpha}{2}, j+1, b^2)\\
	&=F(\frac{\alpha}{2}+1, j+\frac{\alpha}{2}, j, b^2)-F(\frac{\alpha}{2}, j+\frac{\alpha}{2}, j+1, b^2).
\end{align*}
Hence, we finally arrive at
\begin{align}\label{4-10}
	&\quad \partial_{f_2}G_1^\alpha (0, \Omega, 0) h_2\\
	&=(1-\gamma)\alpha b^2\frac{C_\alpha}{2} \sum_{j=1}^\infty b_j \sin(jx)b^{j-1}\frac{\left(\frac{\alpha}{2}+1\right)_{(j-1)}}{(j-1)!} F(\frac{\alpha}{2}, j+\frac{\alpha}{2}, j+1, b^2)\nonumber\\
	&=(1-\gamma)b^2 \sum_{j=1}^\infty  j\Lambda_j(b)b_j\sin(jx).\nonumber
\end{align}

Using the above calculations, one easily obtain
\begin{align}\label{4-11}
	\partial_{f_1}G_2^\alpha (0, \Omega, 0) h_1=-\sum_{j=1}^\infty j\Lambda_j(b)a_j\sin(jx),
\end{align}
and
\begin{align}\label{4-12}
	&\quad \partial_{f_2}G_2^\alpha (0, \Omega, 0) h_2\\
	&=-\sum_{j=1}^{\infty} j\left((1-\gamma)b^{-\alpha}\Theta_j-\Lambda_1(b)\right) b_j \sin(jx).\nonumber
\end{align}
The proof of this lemma is therefore complete.
\end{proof}

We can choose $b$ such that $DG^\alpha(0, \Om, 0)$ is an isomorphism.
\begin{lemma}\label{lm4-4}
	There exists $b_0\in(0,1)$ such that for $b\in(0, b_0)$, $$DG^\alpha(0, \Om, 0):\begin{cases}X_b^k\to Y_b^{k-1},&\quad 0<\alpha<1,\\ X_b^{k+\log}\to Y_b^{k-1},&\quad \alpha=1,\\X_b^{k+\alpha-1}\to Y_b^{k-1},&\quad 1<\alpha<2,\end{cases}$$ is an isomorphism.
\end{lemma}
\begin{proof}
	We choose $b_0$ such that for $0<b<b_0$, $DG^\alpha(0,\Om, 0)$ is invertible. The determinant of $M_j$ is  $$\det(M_j)=-(1-\gamma)b^{-\alpha}\Theta_j\left(\Theta_j+(\gamma-1)b^2\Lambda_1(b)\right)+\Lambda_1(b)\left(\Theta_j+(\gamma-1)b^2\Lambda_1(b)\right)+b^2\Lambda_j(b)^2.$$
	We need the following properties of $\Theta_j$ and $\Lambda_j$, whose proof can be found in \cite{Cas1}, \cite{de1} and \cite{Has}:
	\begin{align*}
		&\Theta_j>0\, \text{for} \,j\geq 2,\\
		&\Theta_j(b)\, \text{is increasing with respect to}\, j,\\
		&\Lambda_j(b)\, \text{ is decreasing with respect to}\, j,\\
		&\Lambda_j(b)\, \text{ is increasing with respect to}\, b,\\
		&\Theta_j=\begin{cases} O(1),&\quad 0<\alpha<1, \\O(\log (j)),&\quad \alpha=1,\\ O(j^{\alpha-1}),&\quad 1<\alpha<2.\end{cases}
	\end{align*}
	In view of the above properties, we choose $b_0\in (0,1)$ such that
	\begin{align*}
		\frac{\Theta_2}{2}\geq b_0^2\Lambda_1(b_0), \quad \frac{(1-\gamma)b_0^{-\alpha}\Theta_2}{4}\geq \Lambda_1(b_0),\quad \frac{(1-\gamma)b_0^{-\alpha}\Theta_2^2}{8}\geq 3b^2_0\Lambda_1(b_0)^2.
	\end{align*}
	Then, for $0<b<b_0$, we have for $j\geq 2$,
	\begin{align*}
		\det(M_j)&=-(1-\gamma)b^{-\alpha}\Theta_j\left(\Theta_j+(\gamma-1)b^2\Lambda_1(b)\right)+\Lambda_1(b)\left(\Theta_j+(\gamma-1)b^2\Lambda_1(b)\right)+b^2\Lambda_j(b)^2\\
		&\leq -(1-\gamma)b^{-\alpha}\Theta_j\left(\Theta_j-b^2\Lambda_1(b)\right)+\Lambda_1(b)\left(\Theta_j+\gamma b^2\Lambda_1(b)\right)+b^2\Lambda_j(b)^2\\
		&\leq -(1-\gamma)b^{-\alpha}\frac{\Theta_j^2}{2}+\Lambda_1(b)\left(\Theta_j+\gamma b^2\Lambda_1(b)\right)+b^2\Lambda_j(b)^2\\
		&\leq -(1-\gamma)b^{-\alpha}\frac{\Theta_j^2}{4}+\gamma b^2\Lambda_1(b)^2+b^2\Lambda_j(b)^2\\
		&\leq -(1-\gamma)b^{-\alpha}\frac{\Theta_j^2}{8}\leq -(1-\gamma)b^{-\alpha}\frac{\Theta_2^2}{4}.
	\end{align*}
    Hence, $M_j$ is invertible for $j\geq 2$.
    Notice that
    $$M_1=\begin{pmatrix} -(1-\gamma)b^2 \Lambda_1(b) & (1-\gamma)b^2 \Lambda_1(b) \\ -\Lambda_1(b) & \Lambda_1(b) \end{pmatrix}=\Lambda_1(b) \begin{pmatrix} -(1-\gamma)b^2  & (1-\gamma)b^2 \\ -1 & 1\end{pmatrix}$$
     is of rank 1. For $g=(g_1,g_2)\in  Y_b^{k-1}$ with $g_1=b^2 d_1\sin(x)+ \sum_{j=1}^\infty c_j\sin(jx)$ and $g_2=\sum_{j=1}^\infty d_j\sin({jx})$, we define its inverse $f=(f_1, f_2)$ with
    $$f^t=\begin{pmatrix} (1-\gamma)b^2\\ 1\end{pmatrix}\frac{d_1}{(1-(1-\gamma)b^2)\Lambda_1(b)} \cos(x)+\sum_{j=2}^\infty \frac{1}{j}M_j^{-1}\begin{pmatrix} c_j\\ d_j \end{pmatrix}\cos(jx).$$
    Then, by the asymptotic behavior of $\Theta_j$, the fact $\Lambda_j\leq \Lambda_1$ and the characterization of $X^{k+\log}$ and $X^{k+\alpha-1}$ (see Proposition 1.1 and Remark 3.2 \cite{Cas4}), it is easy to verify that $f\in X_b^k$ if $ 0<\alpha<1$,  $f\in X_b^{k+\log}$ if $\alpha=1$ and $f\in X_b^{k+\alpha-1}$ if $ 1<\alpha<2$.
\end{proof}

Next, we adjust $\Om$ such that the range of  $G^\alpha$ belongs to $Y^{k-1}_b$.
\begin{lemma}\label{lm4-5}
	There exists
	\begin{equation*}
		\Omega^\alpha(\varepsilon, \boldsymbol f):=\Omega_*^\alpha+\varepsilon\mathcal{R}_\Omega(\varepsilon,\boldsymbol f)
	\end{equation*}
	with
 $$\Omega_*^\alpha=\sum_{n=1}^{N-1} \frac{\alpha C_\alpha(1-\cos(\frac{2\pi n}{N}))}{2\pi\left((-1+\cos(\frac{2\pi n}{N}))^2 +\sin^2(\frac{2\pi n}{N})\right)^{1+\frac{\alpha}{2}}d^{2+\alpha}}$$
 and a continuous function $\mathcal{R}_\Omega(\varepsilon,\boldsymbol f):(-\varepsilon_0,\varepsilon_0)\times V\to \mathbb{R}$, such that $\tilde G^\alpha(\varepsilon,\boldsymbol f):=G^\alpha(\varepsilon,\Omega^\alpha(\varepsilon,\boldsymbol f), \boldsymbol f) $ maps $(-\varepsilon_0,\varepsilon_0)\times V$ to $Y_b^{k-1}$.
	
	Moreover, $\partial_{\boldsymbol f}\mathcal{R}_\Omega(\varepsilon,\boldsymbol f)h:X^k\times X^k\to \mathbb{R}^2$ is continuous.
\end{lemma}
\begin{proof}
	The proof is the same as Lemma \ref{lm3-6}, so we omit it.
\end{proof}

Theorem \ref{thm1'} follows from combining the above lemmas and applying the implicit function theorem.

\section{Existence of traveling solutions}
 The travelling patch pair centered at $(d,0)$ takes the form
\begin{equation*}
	\theta_\varepsilon(\boldsymbol x,t)=\theta_{0,\varepsilon}(\boldsymbol x-tW\boldsymbol e_2)
\end{equation*}
with $W$ some fixed speed, and $\boldsymbol e_2$ the unit vector in $x_2$ direction. According to \eqref{1-1}, we derive
\begin{equation*}
	(\mathbf{u}_0(\boldsymbol x)-W\boldsymbol e_2)\cdot \nabla \theta_{0,\varepsilon}(\boldsymbol x)=0,
\end{equation*}
which yields
\begin{equation*}
	(\mathbf{u}_0(\boldsymbol x)-W\boldsymbol e_2)\cdot \mathbf n(\boldsymbol x)=0,
\end{equation*}
for all $\boldsymbol x$ lies on the interfaces.
According to Biot-Savart law and Green-Stokes formula, when $\alpha=0$, the interfaces satisfies the following equation
\begin{align}\label{5-1}
	0&= H_i^0(\varepsilon, \Omega, \textbf{f})\\
	&=-\pi(1-b^2+\gamma b^2)W\left(\sin(x)-\frac{R_i'(x)}{R_i(x)}\cos(x)\right)\nonumber\\
	&+\frac{1}{2 R_i(x)\varepsilon}\int\!\!\!\!\!\!\!\!\!\; {}-{} \log\left(\frac{1}{ \left(R_i(x)-R_1(y)\right)^2+4R_i(x)R_1(y)\sin^2\left(\frac{x-y}{2}\right)}\right)\nonumber\\
	&\quad\times \left[(R_i(x)R_1(y)+R'_i(x)R'_1(y))\sin(x-y)+(R_i(x)R'_1(y)-R'_i(x)R_1(y))\cos(x-y)\right] dy\nonumber\\
	& -\frac{1-\gamma}{ 2 R_i(x)\varepsilon}\int\!\!\!\!\!\!\!\!\!\; {}-{}\log\left(\frac{1}{ \left(R_i(x)-R_2(y)\right)^2+4R_i(x)R_2(y)\sin^2\left(\frac{x-y}{2}\right)}\right)\nonumber\\
	&\quad\times \left[(R_i(x)R_2(y)+R'_i(x)R'_2(y))\sin(x-y)+(R_i(x)R'_2(y)-R'_i(x)R_2(y))\cos(x-y)\right]dy\nonumber\\
	&- \frac{1}{ R_i(x) \varepsilon} \int\!\!\!\!\!\!\!\!\!\; {}-{}\log\left( \frac{1}{\left| \boldsymbol z_i(x)-d\boldsymbol e_1- Q_{\pi } \left(\boldsymbol z_1(y)-d\boldsymbol e_1\right)\right|}\right)\left[(R_i(x)R_1(y)+R_i'(x)R_1'(y))\sin(x-y-\pi)\right.\nonumber\\
	&\quad \left. +(R_i(x)R_1'(y)-R_i'(x)R_1(y))\cos(x-y-\pi)\right] dy\nonumber\\
	&+ \frac{1-\gamma}{ R_i(x) \varepsilon} \int\!\!\!\!\!\!\!\!\!\; {}-{}\log\left( \frac{1}{\left| \boldsymbol z_i(x)-d\boldsymbol e_1- Q_{\pi } \left(\boldsymbol z_2(y)-d\boldsymbol e_1\right)\right|}\right)\left[(R_i(x)R_2(y)+R_i'(x)R_2'(y))\sin(x-y-\pi)\right.\nonumber\\
	&\quad\left. +(R_i(x)R_2'(y)-R_i'(x)R_2(y))\cos(x-y-\pi)\right] dy.\nonumber
\end{align}
When $0< \alpha<2$, the interfaces obey the following equation
\begin{align}\label{5-2}
	0&= H_i^\alpha(\varepsilon, \Omega, \textbf{f})\\
	&=-\pi(1-b^2+\gamma b^2)W\left(\sin(x)-\frac{R_i'(x)}{R_i(x)}\cos(x)\right)\nonumber\\
	&+\frac{C_\alpha}{R_i(x)\varepsilon|\varepsilon|^\alpha}\int\!\!\!\!\!\!\!\!\!\; {}-{} \frac{1}{\left| \left(R_i(x)-R_1(y)\right)^2+4R_i(x)R_1(y)\sin^2\left(\frac{x-y}{2}\right)\right|^{\frac{\alpha}{2}}}\nonumber\\
	&\quad \times[(R_i(x)R_1(y)+R'_i(x)R'_1(y))\sin(x-y) +(R_i(x)R'_1(y)-R'_i(x)R_1(y))]\cos(x-y)dy\nonumber\\
	& -\frac{(1-\gamma)C_\alpha}{ R_i(x)\varepsilon|\varepsilon|^\alpha}\int\!\!\!\!\!\!\!\!\!\; {}-{}\frac{1}{\left| \left(R_i(x)-R_2(y)\right)^2+4R_i(x)R_2(y)\sin^2\left(\frac{x-y}{2}\right)\right|^{\frac{\alpha}{2}}}\nonumber\\
	&\quad \times [(R_i(x)R_2(y)+R'_i(x)R'_2(y))\sin(x-y)+(R_i(x)R'_2(y)-R'_i(x)R_2(y))]\cos(x-y)dy\nonumber\\
	&- \frac{C_\alpha}{ R_i(x) \varepsilon} \int\!\!\!\!\!\!\!\!\!\; {}-{} \frac{1}{\left| \boldsymbol z_i(x)-d\boldsymbol e_1- Q_{\pi} \left(\boldsymbol z_1(y)-d\boldsymbol e_1\right)\right|^{\alpha}}\left[(R_i(x)R_1(y)+R_i'(x)R_1'(y))\sin(x-y-\pi)\right.\nonumber\\
	&\quad \left. +(R_i(x)R_1'(y)-R_i'(x)R_1(y))\cos(x-y-\pi)\right] dy\nonumber\\
	&+ \frac{(1-\gamma)C_\alpha}{ R_i(x) \varepsilon} \int\!\!\!\!\!\!\!\!\!\; {}-{} \frac{1}{\left| \boldsymbol z_i(x)-d\boldsymbol e_1- Q_{\pi} \left(\boldsymbol z_2(y)-d\boldsymbol e_1\right)\right|^{\alpha}}\left[(R_i(x)R_2(y)+R_i'(x)R_2'(y))\sin(x-y-\pi)\right.\nonumber\\
	&\quad\left. +(R_i(x)R_2'(y)-R_i'(x)R_2(y))\cos(x-y-\pi)\right] dy.\nonumber
\end{align}
{\bf Proof of Theorem \ref{thm2}:} We only sketch the proof. Notice that $H^\alpha$ and  $G^\alpha$ with $N=2$ are almost  identical except for the first term and the signs before the fourth and fifth terms. Thus we can use a similar method as in Sections 3 and 4 to prove the continuity of $H^\alpha(\varepsilon, W, f)$ and $\partial_fH^\alpha(\varepsilon, W, f)$. Moreover, one can verify that for each $W\in\mathbb{R}$, $\partial_fH^\alpha(0, W, 0)$  exactly equals $\partial_f G^\alpha(0,\Om,0)$ and hence $DH^\alpha(0,W,0)$ is an isomorphism under the same assumptions as previous sections. Then we just need to select the appropriate $W$ and apply the implicit function theorem. The  $W_*^\alpha$ chosen is given by
\begin{equation*}
	W_*^\alpha:=\begin{cases}\frac{1}{4\pi d},&\quad \alpha=0,\\ \frac{\alpha C_\alpha}{2\pi (2d)^{1+\alpha}},&\quad 0<\alpha<2.\end{cases}
\end{equation*}
We omit the details.

\phantom{s}
\thispagestyle{empty}

\end{document}